\documentclass{amsart}
\usepackage{verbatim}
\usepackage{rotating} 
\usepackage{amssymb}
\usepackage{mathrsfs,amsfonts,amsmath,amssymb,epsfig,amscd,xy,amsthm,pb-diagram} 

\newtheorem{theorem}{Theorem}[section]
\newtheorem{proposition}[theorem]{Proposition}

\newtheorem{lemma}[theorem]{Lemma}

\newtheorem{corollary}[theorem]{Corollary}
\newtheorem{question}[theorem]{Question}
\newtheorem{definition}[theorem]{Definition}

\theoremstyle{plain}

\theoremstyle{remark}

\newtheorem{remark}[theorem]{Remark}
\newtheorem{example}[theorem]{Example}

\newcommand{\tm}{\tilde{m}}
\newcommand{\tn}{\tilde{n}}

\newcommand{\Q}{{\mathbb Q}}

\newcommand{\Z}{{\mathbb Z}}
\newcommand{\N}{{\mathbb N}}

\newcommand{\cC}{{\mathcal C}}
\newcommand{\cA}{{\mathcal A}}

\def\Frac{\operatorname{Frac}}

\DeclareMathOperator{\Spec}{Spec}

\DeclareMathOperator{\Sym}{Sym}
\DeclareMathOperator{\Gal}{Gal}

\DeclareMathOperator{\tor}{tor}
\DeclareMathOperator{\cha}{char}

\DeclareMathOperator{\ord}{ord}
\DeclareMathOperator{\alg}{alg}

\DeclareMathOperator{\Div}{Div}

\DeclareMathOperator{\id}{id}

\newcommand{\bfx}{{\mathbf x}}
\newcommand{\bfa}{{\mathbf a}}

\newcommand{\bfu}{{\mathbf u}}

\newcommand{\bfm}{{\mathbf m}}
\newcommand{\bfD}{{\mathbf D}}

\newcommand{\bF}{{\mathbb F}}

\newcommand{\cO}{\mathcal{O}}
\newcommand{\cB}{\mathcal{B}}

\newcommand{\cM}{\mathcal{M}}

\newcommand{\scrS}{\mathscr{S}}

\newcommand{\scrX}{\mathscr{X}}

\newcommand{\scrJ}{\mathscr{J}}

\newcommand{\scrA}{\mathscr{A}}

\newcommand{\scrD}{\mathscr{D}}

\newcommand{\scrT}{\mathscr{T}}

\newcommand{\scrM}{\mathscr{M}}

\DeclareMathOperator{\discr}{discr}

\DeclareMathOperator{\tors}{tors}

\author{Jason P.~Bell}
\address{
Jason P.~Bell\\
University of Waterloo\\
Department of Pure Mathematics\\
Waterloo, Ontario, Canada N2L 3G1}
\email{jpbell@uwaterloo.ca}

\author{Khoa D.~Nguyen}
\address{
Khoa D.~Nguyen \\
Department of Mathematics\\
University of British Columbia\\
And Pacific Institute for The Mathematical Sciences\\ 
Vancouver, BC V6T 1Z2, Canada}
\email{dknguyen@math.ubc.ca}
\urladdr{www.math.ubc.ca/\~{}dknguyen}


\begin{document}
	\title[Monogenic Orders in Positive Characteristic]{Some Finiteness Results on Monogenic Orders in Positive Characteristic}
	
	\date{August 30, 2015}
	
	\begin{abstract}
	This work is motivated by the papers 
	\cite{EG1985} 
	and
	\cite{KN-TAMS2015} in which the following
	two problems are solved. Let $\cO$ be a finitely
	generated $\Z$-algebra that is an integrally closed domain of characteristic
	zero, consider the following problems:
	\begin{itemize}
		\item [(A)] Fix $s$ that is integral over $\cO$, describe all
		$t$ such that $\cO[s]=\cO[t]$.		
		
		\item [(B)] Fix $s$ and $t$ that are integral over $\cO$, describe
		all pairs $(m,n)\in\N^2$ such that $\cO[s^m]=\cO[t^n]$.
	\end{itemize} 
	In this paper, we solve these problems 
	and provide a uniform bound
	for a certain ``discriminant form equation''
	that is closely related to Problem (A) 
	when \emph{$\cO$ has characteristic
	$p>0$}. While our general
	strategy roughly follows 
	\cite{EG1985} and \cite{KN-TAMS2015}, many
	new delicate issues arise due to the presence of the Frobenius automorphism
	$x\mapsto x^p$. Recent advances in unit equations over fields of positive 
	characteristic
	together with classical results in
	characteristic zero 
	play an important
	role in this paper.
	\end{abstract}

	\maketitle{}
	
  \section{Introduction}\label{sec:intro}
  Throughout this paper, let $\N$ denote
  the set of positive integers, let
  $\N_0:=\N\cup\{0\}$, 
  and let $p$ be a prime number.  For every power $q>1$ of $p$,
  let $\bF_q$  denote the
  finite field of order $q$. 
  
  Let $\cO$ be a finitely generated $\Z$-algebra that 
  is an integrally closed
  domain with fraction field $K$. 
  Rings of the form $\cO[s]$
  where
  $s$ is integral over $\cO$ and separable over $K$
  are called \emph{monogenic orders} over $\cO$. When $\cha(\cO)=0$,
  certain diophantine aspects of monogenic orders over $\cO$
  have been studied extensively by Gy\H{o}ry, Evertse and other authors
  \cite{Gyory1984}, \cite{EG1985}, \cite{BellHare}, \cite{BEG2013},
  \cite{KN-TAMS2015}. More specifically, when $\cha(\cO)=0$,
  the following two problems are solved in \cite{Gyory1984}, \cite{EG1985},
  \cite{BellHare}, and \cite{KN-TAMS2015}:
  \begin{itemize}
	\item [(A)] Fix $s$ that is integral over $\cO$ and separable over $K$,
	describe all $t$ such that $\cO[s]=\cO[t]$.	
	\item [(B)] Fix $s$ and $t$ that are integral over $\cO$ and separable over $K$, describe
	all pairs $(m,n)\in \N^2$ such that $\cO[s^m]=\cO[t^n]$.
  \end{itemize}

	For Problem (A), Gy\H{o}ry \cite{Gyory1984}
	and Evertse-Gy\H{o}ry \cite{EG1985}   
    prove that there are finitely many 
	elements $t_1,\ldots,t_N$ such that $\cO[t_i]=\cO[s]$ for
	$1\leq i\leq N$, and if $\cO[t]=\cO[s]$ then $t=at_j+b$
	for some $1\leq j\leq N$, $a\in \cO^*$, and $b\in \cO$. Moreover,
	there is a remarkable uniform bound on $N$. 
	After that, Bell and Hare \cite{BellHare}, \cite{BellHare-Corrigendum} study 
	Problem (A) and a weak form
	of Problem (B)	in the special case when $\cO=\Z$ and $s$ and $t$ are 
	algebraic integers satisfying certain properties. The main motivation for
	their work is the so called Pisot-cyclotomic numbers which have
	applications in the study of quasicrystals and quasilattices. Finally,
	in \cite[Theorem~1.4]{KN-TAMS2015}, the second author settles Problem (B) by
	proving that outside certain explicit ``degenerate'' families, there are
	only finitely 
	many $(m,n)\in \N^2$ such that $\cO[s^m]=\cO[t^n]$. Broadly
	speaking, all of these papers use the fact that 
	a linear equation has only
	finitely many non-degenerate solutions taken inside a finitely 
	generated group. Such unit equations play a very important role
	in classical diophantine geometry (see, for instance, \cite{ESS}, \cite[Chapter~5]{BG}, and \cite{EG-UnitBook}).

	The question of what happens when $\cha(\cO)=p$ is natural and  
	interesting on its own. 
   	It is well-known
   	that a na\"ive analogue in characteristic $p$ of many 
		fundamental
   	diophantine
   	problems in characteristic $0$ does not hold and, sometimes, 
   	formulating
   	a correct statement is as important as the proof itself. One of the
   	most spectacular examples is a positive characteristic analogue of the 
   	celebrated Mordell-Lang Conjecture for semi-abelian 
   	varieties \cite{AV1992} 
   	proved by Hrushovski \cite{Hrushovski1996}. Certain aspects
		of Hrushovski's work have been refined
		by results of Moosa-Scanlon \cite{MoosaScanlon}
		and Ghioca \cite{Ghioca-TAMS}. When the ambient semi-abelian 
		variety is a torus, the resulting intersection
		in the Mordell-Lang Conjecture corresponds to
		solutions of certain unit equations
		taken inside a finitely generated group.
		Thanks to
		further work of Voloch, Masser, Derksen, Adamczewski, and the first
		author
		\cite{Voloch1998}, \cite{Masser2004}, \cite{Derksen2007}, 
		\cite{AB2012},
		\cite{DM2012}, rather complete results 
	   on such unit equations (in characteristic $p$) have
		been obtained.
  
	\emph{For the rest of this paper, assume $\cha(\cO)=p$ and $K$
	has transcendence degree at least one over $\bF_p$.} Note
	that the case $K\subset \bar{\bF}_p$ renders both problems 
	(A) and (B)	obvious. While our approach to these problems 
	roughly follows 
	the general
	strategy in \cite{Gyory1984}, \cite{EG1985}, and \cite{KN-TAMS2015}, 
	new delicate algebraic and combinatorial
	issues arise due to the presence of the Frobenius automorphism 
	$x\mapsto x^p$. For Problem (A), at first glance, we might 
	modify the results of Evertse-Gy\H{o}ry by asking
	if there exist finitely many elements $t_1,\ldots,t_N$
	such that every $t$ with $\cO[t]=\cO[s]$
	has the form $t=at_i^{p^m}+b$ for some $1\leq i\leq N$, $m\in \N_0$,
	$a\in \cO^*$, and $b\in \cO$. However, this is not true and 
	a slight modification is needed as illustrated in
	the next example:
	
	\begin{example}\label{ex:A1}  
		Let $\cO=\bF_2[x]$ and let $y$ be a root of $Y^4+x^2Y^2+Y+1=0$.
		Let $s=xy$. For every $m\in \N_0$, write $s_m=xy^{4^m}$.
		It is not
		difficult to show that
		$s_m$ and $s$ have the same discriminant
		over $K$
		and $s_{m+1}\in\cO[s_m]$ for every $m\in\N_0$
		(see Subsection~\ref{subsec:example}, especially
		the proof of Lemma~\ref{lem:OsOzm}),
		we have that $\cO[s_m]=\cO[s]$ for every $m$. 
		It is not hard
		to check that for $i\neq j$, the element $s_i$
		does not have the form 
		$as_j^{2^m}+b$ for some $a\in \cO^*=\{1\}$ and $b\in \cO$.	 		On
		the other hand, we can describe the collection $(s_i)_{i\geq 0}$
		by:
		$$s_i=u_is^{4^i}$$
		where the collection $(u_i=x^{1-4^i})_{i\geq 0}$
		consists of certain powers of the element $x$. A similar 
		and more complicated example will be given in Subsection~
		\ref{subsec:example}.
	\end{example} 
	
	This example shows that, in a certain qualitative sense,
	our result below on Problem (A) is optimal (also see Remark~\ref{rem:example}). Similar
	to results by Evertse-Gy\H{o}ry \cite{EG1985},
	our bound depends uniformly on (certain invariants of) 
	the ring $\cO$ and the degree $[K(s):K]$ as follows. By a 
	theorem of 
	Roquette \cite{Roquette}, the unit group $\cO^*$ is
	finitely generated. Let $q(K)$ be the power of
	$p$ such that $K\cap\bar{\bF}_p=\bF_{q(K)}$, hence
	$|\cO^*_{\tors}|=q(K)-1$. Let $V=\Spec(\cO)$
	and fix a choice of a projective normal
	scheme $\bar{V}$ over $\bF_{q(K)}$
	together with an open embedding from $V$
	to $\bar{V}$. Let $M_K$ be the set of discrete
	valuations on $K$ associated to the 
	Weil divisors of $\bar{V}$ (see \cite[pp.~130]{HartshorneAG})
	and let $S$ be the finite subset of $M_K$
	corresponding to the Weil divisors contained in
	$\bar{V}\setminus V$. For $v\in M_K$, let $n_v$
	denote the degree
	of the Weil divisor corresponding $v$ (see \cite{HartshorneAG}). 
	We have the following properties:
	\begin{itemize}
		\item [(i)] For every $a\in K^*$, $v(a)=0$
		for all but finitely many $v\in M_K$ and 
		$\sum_{v\in M_K}n_v v(a)=0$ (base change
		from $\bF_{q(K)}$ to $\bar{\bF}_p$
		and use \cite[Exercise~II.6.2(d)]{HartshorneAG}).
		
		\item [(ii)] $\cO=\{a\in K:\ v(a)\geq 0\ \text{for every
		 $v\in M_K\setminus S$}\}$ (see \cite[pp.~132]{HartshorneAG}).
		 
		 \item [(iii)] $\bF_{q(K)}^*=\{a\in K^*: v(a)=0\ \text{for every $v\in M_K$}\}$ (see \cite[pp.~122]{HartshorneAG}).
	\end{itemize}
	
	For every element $\alpha$
    that is separable over $K$, we define the 
    discriminant of $\alpha$ over $K$ by:
    $$\discr_K(\alpha)=\prod_{1\leq i<j\leq d} 
    (\alpha_i-\alpha_j)^2$$
  where $\alpha_1,\ldots,\alpha_d$ are all the conjugates
  of $\alpha$ over $K$.
    Our first main result provides an answer
    to Problem (A) with a uniform bound on $q(K)$, $|S|$,
    and $[K(s):K]$:
	\begin{theorem}\label{thm:A}
		Let $s$ be integral over $\cO$ and separable over $K$
		of degree $d=[K(s):K]\geq 2$ and let $D=\discr_K(s)$. 
		There are 
		$$N\leq q(K)^{d^6}+
		\left(\exp(18^{10})p^{3d^4|S|}\log_p q(K)\right)^{d^3}$$ 
		elements $t_1,\ldots,t_N$
		satisfying the following conditions:
		\begin{itemize}
			\item [(a)] $\cO[t_i]=\cO[s]$ for $1\leq i\leq N$.
			\item [(b)] If $\cO[t]=\cO[s]$
			then $t=at_i^{q}+b$
			where $1\leq i\leq N$, $q\geq 1$ is
			a power of 
			$p$, $a,b\in K$
			such that $\displaystyle\frac{a^{d(d-1)}}{D^{1-q}}\in\cO^*$.
		\end{itemize}
	\end{theorem}
	
	\begin{remark}\label{rem:example}
		We will prove a slightly more precise result, see
		Theorem~\ref{thm:A precise} and Remark~\ref{rem:A precise}.
		In characteristic zero, we have the form
		$t=at_i+b$ instead of part (b) (see \cite{EG1985} or
		\cite[pp.~6--9]{KN-TAMS2015}) with $a\in \cO^*$ and, hence,
		$b$ must automatically be in $\cO$. On the other hand,
		we will construct an example in Subsection~
		\ref{subsec:example} to show that
		in characteristic $p$, it is not always possible to
		have $t=at_i^q+b$ as in Theorem~\ref{thm:A} with
		the further restriction that $b$ is in $\cO$.
	\end{remark}
	
	It is well-known that the ``monogenic order equation"
	$\cO[t]=\cO[s]$ is closely related to the problem 
    of solving for integral elements with a given 
    discriminant (see, for example, \cite{Gyory1984} 
    and \cite{EG1985}).  Now for the equation
    $\discr_K(t)=\delta$ with a given $\delta$,
    we may consider a more general problem by 
    defining $\tilde{\cO}=\cO[1/\delta]$ and solving for $t$
    that is integral over $\tilde{\cO}$ with $\discr_K(t)\in \tilde{\cO}^*$. This motivates our next result.
    
    For a finite subset $T\subset M_K$
    containing $S$, the set of $T$-integral elements of $K$ is
    defined to be:
    $$\cO_{K,T}:=\{a\in K: v(a)\geq 0\ \text{for every $v\in M_K
    \setminus T$}\}.$$
    In particular $\cO_{K,S}=\cO$. Theorem~\ref{thm:A} immediately
    gives the following:
    \begin{corollary}\label{cor:A}
    	Let $E$ be a separable extension of degree $d$
    	over $K$
    	and let $T$ be a finite subset of $M_K$ containing $S$. 
    	Let $N(d)$ denote the number of subgroups of the 
    	symmetric group $\Sym(d)$. Then there are 
    	$$N\leq
    	N(d) \left(q(K)^{d^6}+
		\left(\exp(18^{10})p^{3d^4|T|}\log_p q(K)\right)^{d^3}\right)$$ 
    	elements
    	$t_1,\ldots,t_N$ in $E$ satisfying the following
    	conditions:
    	\begin{itemize}
			\item [(a)] For $1\leq i\leq N$, $t_i$ is integral over $\cO_{K,T}$ and $\discr_K(t_i)\in \cO_{K,T}^*$.
			\item [(b)] If $t\in E$ is integral
			over $\cO_{K,T}$ and  
			$\discr_K(t)\in \cO_{K,T}^*$
			then $t=at_i^q+b$ where $1\leq i\leq N$, $q\geq 1$
			is a power of $p$, $a\in \cO_{K,T}^*$, 
			and $b\in \cO_{K,T}$.    	
    	\end{itemize}  
    \end{corollary}	
	
	\begin{remark}
		If $\discr_K(t_i)\in \cO_{K,T}^*$
		and $t=at_i^q+b$ as above then $\discr_K(t)\in\cO_{K,T}^*$.
		Therefore Corollary~\ref{cor:A}
		is optimal (at least qualitatively).
	\end{remark}
	
	We now address Problem (B) where similar
	issues arise due to the Frobenius automorphisms. Note that 
	when $\cha(\cO)=0$,
	the main result of \cite[Theorem~1.3]{KN-TAMS2015} implies that
	the set $\{(m,n):\ \cO[s^m]=\cO[t^n]\}$ is the union of a 
	finite
	set and at most finitely many ``progressions''
	of the form $\{(km_0,kn_0):\ k\in\N\}$
	for some $(m_0,n_0)$. However, when $\cha(\cO)=p$, if
	$\cO[s^m]=\cO[t^n]$ then $\cO[s^{pm}]=\cO[t^{pn}]$; therefore
	infinite sets of the form $\{(p^km_0,p^kn_0):\ k\in\N_0\}$ will 
	arise. 
	We now give further examples that more complicated sets could 
	appear.
	
	For $s$ and $t$
	that are integral over $\cO$ and separable over $K$, we denote:
	$$\cM(\cO,s,t):=\{(m,n)\in\N^2:\ \cO[s^m]=\cO[t^n]\}.$$
	\begin{example}\label{ex:1st}
		Suppose for some $(m_0,n_0)\in\N^2$ and some power $q>1$ of 
		$p$,
		we have
		$\cO[s^{m_0}]=\cO[t^{n_0}]=\cO[t^{qn_0}]$. Then we have:
		$$\{(q^im_0,q^jn_0):\ i,j\in\N_0\}\subseteq \cM(\cO,s,t).$$
	\end{example}
	
	\begin{example}\label{ex:2nd}
    Here is an explicit example where a set of the form
    $$\{(c_1q^i+c_2q^j,c_3q^i+c_4q^j):\ i,j\in \N_0\}$$
	for some $c_1,c_2,c_3,c_4\in\N$ and some power $q>1$
	of $p$ is contained 
	in $\cM(\cO,s,t)$. Note that
	the set in Example~\ref{ex:1st} is a special case in which
	$c_1=m_0$, $c_2=c_3=0$, and $c_4=n_0$. Let $p=7$, $L=\bF_7(x,y)$,
	$\cO=\bF_7[x+y,xy]$, $K=\Frac(\cO)$, $s=x$, $t=3x+2y$. For 
	$i,j\in\N$ and $m=n=7^i+7^j$, it is
	not difficult to show directly that
	$s^m\in \cO[t^n]$ and $t^n\in \cO[s^m]$; in other words
	$(m,n)\in\cM(\cO,s,t)$.
	\end{example}

	We need the following:
	\begin{definition}\label{def:non-degenerate}
	Let $a_1,a_2,a_3,a_4\in\Q$, define:
		$$F(q;a_1,a_2,a_3,a_4):=
		\{(a_1q^i+a_2q^j,a_3q^i+a_4q^j):\
		i,j\in\N_0\}\subset \Q^2.$$
	\end{definition}

	We will obtain a list of
	unit equations from the equation $\cO[s^m]=\cO[t^n]$,
	then
	the sets $F(q;a_1,\ldots,a_4)$ in
	 Definition~\ref{def:non-degenerate}
	correspond to the non-degenerate solutions. Degenerate
	solutions will correspond to the following sets:
	
	\begin{definition}\label{def:degenerate}
		Let $s$ and $t$ be integral over $\cO$ and separable
		over $K$. Define:
		$$\cA_{\cO,s,t}=\{(m,n)\in\N^2:\ \frac{s^m}{t^n}\in\cO^*\},$$
  $$\cB_{\cO,s,t}=\{(m,n)\in\N^2:\ [K(t^n):K]=2\ \text{and 
   }\frac{s^m}{\sigma(t^n)}\in\cO^{*}\}$$
   where $\sigma$ in the definition of $\cB_{\cO,q,r}$ is
   the nontrivial automorphism of the quadratic extension
   $K(t^n)/K$. Finally, we define:
   $$\cC_{\cO,s,t}=\{(m,n)\in\N^2:\ s^mt^n\in\cO^{*}\}.$$
	\end{definition}

	    By the same arguments in \cite[pp.~2--3]{KN-TAMS2015}, 
	    we have that each of
		the sets $\cA(\cO,s,t)$, $\cB(\cO,s,t)$, and  $\cC(\cO,s,t)$
		is a subset of $\cM(\cO,s,t)$.  	
	We can now state the second main result of this paper. As in 
	\cite[pp.~2]{KN-TAMS2015},
	for simplicity
	we \emph{assume the very mild condition that
	$\{s^n,t^n:\ n\in\N\}\cap\cO=\emptyset$.} If $s^n\in \cO$ or $t^n\in \cO$
	for some $n$, Problem (B) becomes much easier  
	and is treated
	in Section~\ref{sec:addendum} (compare \cite[Section~5]{KN-TAMS2015}). 
		
	\begin{theorem}\label{thm:B}
		Let $s$ and $t$ be integral over $\cO$, separable over $K$,
		and satisfy $\{s^n,t^n:\ n\in\N\}\cap\cO=\emptyset$. 
		Then the set 
		$$\cM(\cO,s,t)\setminus \left(\cA(\cO,s,t)\cup\cB(\cO,s,t)\cup\cC(\cO,s,t)\right)$$
		is contained in a finite union of sets of the form
		$$F(q;c_1,c_2,c_3,c_4)=
		\{(c_1q^i+c_2q^j,c_3q^i+c_4q^j):\ i,j\in\N_0\}$$
		for some power $q>1$ 
		of $p$ and  $c_1,c_2,c_3,c_4\in\Q$. 
	\end{theorem}

		Example~\ref{ex:2nd} shows that \emph{in general} 
		we cannot improve 
		Theorem~\ref{thm:B} in the sense that
		the sets $F(q;c_1,c_2,c_3,c_4)$ 
		appearing in 
		$\cM(\cO,s,t)\setminus\
		(\cA(\cO,s,t)\cup\cB(\cO,s,t)\cup\cC(\cO,s,t))$	
		consist of pairs $(m,n)$ where each of $m$ and $n$
		is a linear combination of $q^i$ and $q^j$
		for 2 parameters $i,j\in \N$. 
		While we expect the number of sets 
	    $F(q;c_1,c_2,c_3,c_4)$
	    in Theorem~\ref{thm:B} could be bounded uniformly
	    in terms of $q(K)$, $|S|$ and $[K(s,t):K]$,
	    our proof does not seem to yield this. 
	    The proofs of Theorem~\ref{thm:A},
	    Corollary~\ref{cor:A}, and Theorem~\ref{thm:B}
	    are not effective.

		The organization of this paper is as follows.
		In the next section, we introduce finiteness results
		for unit equations in both zero and positive 
		characteristics. After that, we prove Theorem~\ref{thm:A},
		Corollary~\ref{cor:A}, and Theorem~\ref{thm:B}. The
		last section addresses the easy case
		of Problem (B) when 
		$\{s^n,t^n:\ n\in\N\}\cap \cO\neq\emptyset$.

	{\bf Acknowledgments.} We wish to thank Professors 
	Jan-Hendrik Evertse, Dragos Ghioca, K\'alm\'an Gy\H{o}ry,
	and David Masser  
	for useful discussions. The second author is grateful
	to Professors Evertse and Gy\H{o}ry for answering many
	questions and sharing the draft of their upcoming book on 
	the topics (in characteristic zero) presented in this paper.

	\section{Unit Equations}\label{sec:unit} 
	Let $n\in\N$, a solution $(x_1,\ldots,x_n)$ of the equation
	$a_1X_1+\ldots+a_nX_n=1$ with non-zero parameters $a_i$'s is
	called \emph{non-degenerate} if no subsums vanish. In other words,
	there is no proper subset $\emptyset\neq J\subset\{1,\ldots,n\}$
	such that $\sum_{j\in J} a_jx_j=0$. 
	
	We start with a celebrated result on unit equation in 
	characteristic 
	zero proved by Evertse, Schlickewei, and Schmidt \cite{ESS}:
	
	\begin{proposition}\label{prop:unit char 0}
		Let $L$ be a field of characteristic $0$ and let $G$
		be a finitely generated subgroup of $L^*$ having
		rank $r$. Let $n\in\N$ and
		$a_1,\ldots,a_n\in L^*$, then the number of
		non-degenerate solutions $(x_1,\ldots,x_n)\in G^n$
		of the equation:
		$$a_1X_1+\ldots+a_nX_n=1$$
		is at most $\exp\left((6n)^{3n}(nr+1)\right)$. 
	\end{proposition}

	Now we consider unit equations in positive characteristic 
	where, as usual, subtle issues arise due to the presence of the 
	Frobenius automorphism.
	For the rest of this section, let $L$ be a  
	field of characteristic $p$ 
	and let $G$ be a subgroup
	of $L^*$. 
	The radical of $G$ in $L$ is defined to be: 
	$$\sqrt[L]{G}:=\{\gamma\in L:\ \gamma^m\in G\ \text{for some $m\in\N$}\}.$$
	\emph{Assume that the group $\sqrt[L]{G}$
	is finitely generated.} When 
	$L$ is finitely generated over $\bF_p$,  finite 
	generation of
	$\sqrt[L]{G}$ is equivalent to that of $G$.
	We have the 
	following result by
	Derksen and Masser:
	
	\begin{proposition}\label{prop:DM-PLMS}
	Let $a_1,\ldots,a_n\in L^*$. Consider the equation $a_1x_1+\ldots+a_nx_n=1$ with $(x_1,\ldots,x_n)\in G^n$. Then there is
	a finite set $\scrS$ (contained in $\bar{L}^*$) such that every non-degenerate 
	solution 
	$(x_1,\ldots,x_n)$ has the form:
	  $$x_k=\alpha_{k,0} \alpha_{k,1}^{p^{i_1}}\ldots \alpha_{k,n-1}^{p^{i_{n-1}}}\ \text{for $k=1,\ldots,n$}$$
	  for some $i_1,\ldots,i_{n-1}\in\N_0$, and 
	  $\alpha_{k,j}\in \scrS$ for $0\leq j\leq n-1$.
	\end{proposition}
	\begin{proof}
		This follows from \cite[Theorem~3]{DM2012} where
		the form of the $x_k$'s follows
		from the dehomogenization of points in the set 
		$(\pi_0,\ldots,\pi_h)_p$ (with $h\leq n-1$) given
		in \cite[Theorem~3]{DM2012}. 
		Earlier versions
		of this result were obtained by 
		Moosa-Scanlon \cite{MoosaScanlon} and 
		Ghioca \cite{Ghioca-TAMS} (see the
		comments in \cite[pp.~1050]{DM2012}).
	\end{proof}
	
	We use Proposition~\ref{prop:DM-PLMS} to obtain
	the following:
	\begin{proposition}\label{prop:refined n}
		Consider the equation $x_1+\ldots+x_n=1$ with 
		$(x_1,\ldots,x_n)\in G^n$. Then there is a positive
		integer $C$ and a finite
		set $\scrS'$ (contained in $\bar{L}^*$) such that for every non-degenerate solution
		$(x_1,\ldots,x_n)$, we have:
		$$x_k^{p^C}=\alpha_{k,1}^{p^{i_1}}\ldots \alpha_{k,n-1}^{p^{i_{n-1}}}\ \text{for $k=1,\ldots,n$}$$
	  for some $i_1,\ldots,i_{n-1}\in\N_0$, and 
	  $\alpha_{k,j}\in \scrS'$ for $1\leq j\leq n-1$.
	\end{proposition}
	
	In other words, this says that 
	for every non-degenerate solution of the equation $x_1+\ldots+x_n=1$, after raising
	to some $p^C$-th power, 
	 we can 
	omit the ``translation by 
	$(\alpha_{1,0},\ldots,\alpha_{n,0})$''
	in the description given in Proposition~\ref{prop:DM-PLMS}. We
	refer the readers to Question~\ref{q:C=0} for further 
	refinement in this direction. To
	prove Proposition~\ref{prop:refined n}, we need 
	the following simple lemma:
	
	\begin{lemma}\label{lem:e_iu_i}
	Let $N\geq 1$ and let $e_1,\ldots,e_N$ be integers. 
	There exist $C_1$
	depending on $N$ and the $e_i$'s 
	such that the following holds. For every $u_1,\ldots,u_N\in\Z$
	satisfying two conditions:
	\begin{itemize}
		\item [(i)] $\sum_{i=1}^N e_ip^{u_i}\in \Z\setminus\{0\}$,
		\item [(ii)] there does not exist a non-empty proper
		subset $J\subset\{1,\ldots,N\}$
		such that $\sum_{i\in J} e_ip^{u_i}=0$,
	\end{itemize}
	we have that $u_i+C_1\geq 0$ for every $1\leq i\leq N$.
	\end{lemma}
	\begin{proof}
	Induction on $N$, the case $N=1$ gives that 
	$u_1+\ord_p(e_1)\geq 0$. Now let $N\geq 2$ and assume that 
	the lemma holds
	for every smaller value of $N$.
	 Since $|\sum_{i=1}^N e_ip^{u_i}|\geq 1$,
	there is a lower bound $-C_2$ on $\max\{u_1,\ldots,u_N\}$.
	Say $u_N=\max\{u_i\}_i$, then apply the induction
	hypothesis for
	$e_1p^{u_1+C_2}+\ldots+e_{n-1}p^{u_{N-1}+C_2}$.
	\end{proof}
	
	\begin{remark}\label{rem:e_iu_i}
	When some $e_i$ in Lemma~\ref{lem:e_iu_i}
	is zero, the lemma is vacuously true
	for any choice $C_1$ since there do not exist
	$u_1,\ldots,u_N$ satisfying the conditions (i) and (ii).
	If this is the case, we will simply choose $C_1=0$.
	\end{remark}
	
	\begin{proof}[Proof of Proposition~\ref{prop:refined n}]
		Let $\scrS$ be a finite set 
		satisfying the conclusion of 
		Proposition~\ref{prop:DM-PLMS} for the
		unit equation $x_1+\ldots+x_n=1$. Let $\Gamma$
		be the group generated by $\scrS$
		and let $r$ denote its rank. For every
		$x\in \Gamma$, let $\bar{x}$ denote its image in
		$\Gamma/\Gamma_{\tors}$. Let $g_1,\ldots,g_r\in \Gamma$ 
		such that $\{\bar{g}_i:\ 1\leq i\leq r\}$
		is a
		basis of $\Gamma/\Gamma_{\tors}$. Let $E_1,\ldots,E_r$
		be the functions from $\Gamma$ to $\Z$ satisfying:
		$$\bar{x}=\prod_{\ell=1}^r \bar{g}_{\ell}^{E_{\ell}(x)}$$
		for every $x\in \Gamma$. Define: 
		\begin{equation}\label{eq:tilde scrS}
		\tilde{\scrS}:=\left\{\prod_{\ell=1}^r g_{\ell}^{d_{\ell}}:\  
		d_{\ell}\in \{0\}\cup E_{\ell}(\scrS)\ \text{for $1\leq \ell\leq r$}\right\}. 
		\end{equation}
		Then the desired set
			$\scrS'$ is defined as follows:
		\begin{equation}\label{eq:scrS'}	
		\scrS':=\{\alpha\beta:\ \alpha\in \Gamma_{\tors},\ 
			\beta\in\tilde{\scrS}\}.
		\end{equation}
		
		Let $(x_1,\ldots,x_n)$ be a non-degenerate solution
		of the given unit equation. For every $m\in \N_0$,
		$(x_1^{p^m},\ldots,x_n^{p^m})$
		is also a non-degenerate solution. By the
		definition of $\scrS$, we have:
		\begin{equation}\label{eq:p^m}
		x_k^{p^m}=\alpha_{k,0,m}\alpha_{k,1,m}^{p^{i_{1,m}}}\ldots \alpha_{k,n-1,m}^{p^{i_{n-1,m}}}\ \text{for $k=1,\ldots,n$
		and for $m\in \N_0$}
		\end{equation}
	  for some $i_{1,m},\ldots,i_{n-1,m}\in\N_0$, and 
	  $\alpha_{k,j,m}\in \scrS$ for $0\leq j\leq n-1$. Since
	  $\scrS$ is finite, we may assume that \eqref{eq:p^m}
	  holds for infinitely many $m$ for \emph{one} tuple
	  $(\alpha_{k,j})$. In other words, there
	  is an infinite subset $\cM_1$ of $\N_0$
	  such that:
	  \begin{equation}\label{eq:m in cM1}
		x_k^{p^m}=\alpha_{k,0}\alpha_{k,1}^{p^{i_{1,m}}}\ldots \alpha_{k,n-1}^{p^{i_{n-1,m}}}\ \text{for $k=1,\ldots,n$
		and for $m\in \cM_1$}
		\end{equation}
	  for some $i_{1,m},\ldots,i_{n-1,m}\in\N_0$, and 
	  $\alpha_{k,j}\in \scrS$ for $0\leq j\leq n-1$. Write:
	  \begin{equation}\label{eq:xk bkl}
	  \bar{x}_k=
	  \prod_{\ell=1}^r \bar{g}_{\ell}^{b_{k,\ell}}\ \text{for
	  $1\leq k\leq n$}.
	  \end{equation}
	\begin{equation}\label{eq:alphakj ekjl}	
	\bar{\alpha}_{k,j}=\prod_{\ell=1}^r g_{\ell}^{e_{k,j,\ell}}\ 
	\text{for $1\leq k\leq n$ and $0\leq j\leq n-1$}.
	\end{equation}
	Therefore
	$p^mb_{k,\ell}=e_{k,0,\ell}+e_{k,1,\ell}p^{i_{1,m}}+\ldots+e_{k,n-1,\ell}p^{i_{n-1,m}}$ for $1\leq k\leq n$
	and $m\in\cM_1$. 
	
	For each $1\leq k\leq n$ and $1\leq \ell\leq r$, we claim that 
	for all but finitely many $m\in\cM_1$, $p^mb_{k,\ell}$ is a 
	(not necessarily proper) subsum
	of $e_{k,1,\ell}p^{i_{1,m}}+\ldots+e_{k,n-1,\ell}p^{i_{n-1,m}}.$
	Indeed, there is nothing to prove when $e_{0,\ell}=0$; when
	$e_{0,\ell}\neq 0$, this follows from
	Proposition~\ref{prop:unit char 0}. 
	We now exclude those finitely many $m$'s from $\cM_1$ as in the 
	claim and
	let $\cM_2$ denote the resulting infinite set.

	Let $\Lambda$ be the set of pairs $(k,\ell)$
	with $1\leq k\leq n$ and $1\leq \ell\leq r$ such
	that $b_{k,\ell}\neq 0$.
	For every $(k,\ell)\in\Lambda$ and for
	every (non-empty) subset $J$ of $\{1,\ldots,n-1\}$
	consider the set $M(k,\ell,J)\subseteq \cM_2$ of $m$ satisfying 
	the 
	following conditions:
	\begin{itemize}
		\item [(i)] $p^mb_{k,\ell} =\displaystyle\sum_{j\in J} e_{k,j,\ell}p^{i_{j,m}}$.
		\item [(ii)] The sum $\displaystyle\sum_{j\in J} e_{k,j,\ell}p^{i_{j,m}}$ has no vanishing proper subsum. 
	\end{itemize}
	By the above claim, we have:
	$\bigcup_{J} M(k,\ell,J) =\cM_2$
	where $J$ ranges over all the non-empty subsets
	of $\{1,\ldots,n-1\}$. Obviously, this gives:
	$$\bigcap_{(k,\ell)\in \Lambda} 
	\left(\bigcup_J M(k,\ell,J)\right)=\cM_2.$$
	Therefore it is possible to choose a non-empty subset
	$\scrJ_{k,\ell}$ of $\{1,\ldots,n-1\}$
	for each $(k,\ell)\in\Lambda$
	such that the set
	$$\cM_3:=\bigcap_{(k,\ell)\in \Lambda} 
	M(k,\ell,\scrJ_{k,\ell})$$
	is infinite.

	Pick one $m'\in \cM_3$. By the definition of $\cM_3$
	and the sets $M(k,\ell,\scrJ_{k,\ell})$, we have the following:
	\begin{itemize}
		\item [(i)] $b_{k,\ell} 
		=\displaystyle\sum_{j\in \scrJ_{k,\ell}} e_{k,j,\ell}p^{i_{j,m'}-m'}$ for every $(k,\ell)\in\Lambda$.
		\item [(ii)] The sum $\displaystyle\sum_{j\in \scrJ_{k,\ell}} e_{k,j,\ell}p^{i_{j,m'}-m'}$ has no vanishing proper subsum
		for every $(k,\ell)\in\Lambda$. 
	\end{itemize}
	
	Now we let $\Omega$
	range over all non-empty
	subset of $\{1,\ldots,n\}\times \{1,\ldots,r\}$,
	let $(J_{k,\ell})_{(k,\ell)\in \Omega}$
	range over all possible
	$|\Omega|$-tuple
	of non-empty subsets of $\{1,\ldots,n-1\}$, and let $C$ be
	the maximum of all the $C_1$'s
	obtained when applying Lemma~\ref{lem:e_iu_i}
	for the tuples $(e_{k,j,\ell})_{j\in J_{k,\ell}}$ 
	for $(k,\ell)\in \Omega$
	(see Remark~\ref{rem:e_iu_i}). 
	
	We have:
	\begin{equation}\label{eq:p^Cbkl}
	p^Cb_{k,\ell}=\displaystyle\sum_{j\in \scrJ_{k,\ell}} e_{k,j,\ell}p^{i_{j,m'}-m'+C}=\displaystyle\sum_{j=1}^{n-1} f_{k,j,\ell}p^{i_{j,m'}-m'+C} 
	\end{equation}
	for every $1\leq \ell\leq r$
	where $f_{k,j,\ell}:=e_{k,j,\ell}$ 
	if $(k,\ell)\in\Lambda$ and $j\in \scrJ_{k,\ell}$;
	otherwise $f_{k,j,\ell}:=0$. This implies
	that $f_{k,j,\ell}\in \{0\}\cup E_{\ell}(\scrS)$
	for every $1\leq k\leq n$, $1\leq j\leq n-1$,
	and $1\leq \ell\leq r$. Hence the element:
	$$\beta_{k,j}:=\prod_{\ell=1}^r g_{\ell}^{f_{k,j,\ell}}\   \text{belongs to $\tilde{\scrS}$ for $1\leq k\leq n$ and $1\leq j\leq n-1$}.$$
	Let $i'_j=i_{j,m'}-m'+C$ which is non-negative for $1\leq j\leq n-1$ by the definition of $C$. By \eqref{eq:xk bkl},
	\eqref{eq:p^Cbkl}, and the definition of the 
	$\beta_{k,j}$'s, we have:
	$$x_k^{p^C}\equiv \beta_{k,1}^{p^{i'_1}}\ldots\beta_{k,n-1}^{p^{i'_{n-1}}}\ \text{in $\Gamma/\Gamma_{\tors}$}$$
	for $1\leq k\leq n$. Since $\Gamma_{\tors}$
	is a finite cyclic group whose order is
	relatively prime to $p$, for $1\leq k\leq n$ there
	is $\zeta_k\in\Gamma_{\tors}$
	such that:
	$$x_k^{p^C}=(\zeta_k\beta_{k,1})^{p^{i'_1}}\beta_{k,2}^{p^{i'_2}}\ldots\beta_{k,n-1}^{p^{i'_{n-1}}}.$$
	This shows that the pair $(C,\scrS')$ satisfies the
	desired conclusion.
	\end{proof}
	
	When $n=2$, we have a more precise result:
	\begin{proposition}\label{prop:refined 2}
		Let $r$ denote the rank of $G$ and consider the 
		equation $x+y=1$ with $(x,y)\in G^2$.  We have: 
		\begin{itemize}
		\item [(a)] There exists a finite subset $\scrX$
		of $L^*\times L^*$ of size at most $p^{2r}-1$ 
		such that every solution $(x,y)\in G^2\setminus \bar{\bF}_p^2$ has the form
		$x=x_0^{p^k}$ and $y=y_0^{p^k}$
		for some $(x_0,y_0)\in \scrX$ and $k\in \N_0$.
		
		\item [(b)] There exists a finite subset $\scrX'$
		of $L^*\times L^*$ of size at most $p^{2r}$ 
		such that every solution $(x,y)\in G^2$ has the form
		$x=x_0^{p^k}$ and $y=y_0^{p^k}$
		for some $(x_0,y_0)\in \scrX'$ and $k\in \N_0$.
		\end{itemize}
	\end{proposition}
	\begin{proof}
	We prove (a) first. Write $H=\sqrt[L]{G}$ which is finitely 
	generated by our 
	assumption.  Note that $H/G$ is a torsion abelian group and 
	so the rank of $H$ 
	is $r$.  We first consider solutions to the 
	equation $x+y=1$ with $(x,y)\in H\times H$.  We have that 
	$H\cong \mathbb{Z}^r \times \mathbb{F}_q^*$ where  
	$\mathbb{F}_q=L\cap\bar{\bF}_p$.  Since the Frobenius map is 
	surjective on $\mathbb{F}_q^*$ we have that $H/H^p\cong 
	\left(\mathbb{Z}/p\mathbb{Z}\right)^r$.  Let 
	$1=\epsilon_0,\ldots ,\epsilon_{p^r-1}$ be a set of coset 
	representatives for $H/H^p$.  Let 
	$(i,j)\in \{0,1,\ldots ,p^r-1\}^2$ with $(i,j)\neq (0,0)$.  
	We now consider all solutions to the equation $x+y=1$ with 
	$(x,y)\in H^p\epsilon_i \times H^p\epsilon_j$.  
	Observe that at least one of $\epsilon_i,\epsilon_j$ cannot be 
	in $L^{\langle p\rangle}:= \{a^p \colon a\in L\}$. Otherwise, 
	we would have 
	$\epsilon_i =a^p$ and $\epsilon_j=b^p$ for some $a,b\in L$.  
	Since $H=\sqrt[L]{G}$ we would then have 
	$a,b\in H$ and so $\epsilon_i,\epsilon_j\in H^p$ contradicting
	our choice that $(i,j)\neq (0,0)$.  
	
	We assume that 
	$\epsilon_i\not\in L^{\langle p \rangle}$, the case when $\epsilon_j\not\in L^{\langle p\rangle}$ 
	can be handled similarly.  Then the sum
	$L^{\langle p\rangle}+ L^{\langle p\rangle}\epsilon_i$ is 
	direct.  On the other hand, 
	if 
	$L^{\langle p\rangle}+ 
	L^{\langle p\rangle}\epsilon_i+ L^{\langle p\rangle}\epsilon_j$ 
	is direct then there cannot be any solutions to the equation 
	$x+y =1$ with $(x,y)\in H^p \epsilon_i \times H^p \epsilon_j$.  
	Thus it suffices to consider the case when 
	$\epsilon_j = a^p + b^p \epsilon_i$ with $a,b\in L$.  
	Note that the pair $(a,b)$, if exists, is unique.
	Write $x=x_1^p\epsilon_i$ and $y=y_1^p\epsilon_j$, then
	the equation $x+y=1$ gives:
	$$x_1^p\epsilon_i+y_1^p\epsilon_j=(ay_1)^p+(x_1^p+b^py_1^p)
	\epsilon_i=1.$$
	This gives $ay_1=1$ and $x_1+by_1=0$ since $L^{\langle p\rangle}+ L^{\langle p\rangle}\epsilon_i$ is 
	direct. Therefore $(x_1,y_1)$ and, hence, $(x,y)$
	are uniquely determined. Overall, we have at most $p^{2r}-1$
	solutions to the equation $x+y=1$ with  $(x,y)\in H\times H$ 
	and $(x,y)\not \in H^p\times H^p$.  Let $M\leq p^{2r}-1$	
	and let $(x_i,y_i)\in H\times H$ with $i=1,\ldots ,M$ 
	denote the collection of all such solutions.  Then if 
	$(x,y)\in H\times H$ is a solution to $x+y=1$ with $x$ and $y$ 
	not algebraic over $\mathbb{F}_p$ then there is some largest 
	$m$ such that $(x,y)\in H^{p^m}\times H^{p^m}$.  Therefore 
	there must exist some $i\in\{1,\ldots,M\}$ such 
	that $x=x_i^{p^m}$ and $y=y_i^{p^m}$.  
	
	We now consider the solutions $(x,y)$ in $G\times G$.  
	For each $i\in \{1,\ldots ,M\}$, consider the set $N_i$ of all 
	nonnegative integers $n$ for which $x_i^{p^n}$ and $y_i^{p^n}$ 
	are both in $G$.  This set is either empty or has some least 
	element $n_i$.  Letting $I$ denote the set of 
	$i\in \{1,\ldots ,M\}$ for which $N_i$ is nonempty and then 
	letting $\scrX=\{(x_i^{p^{n_i}},y_i^{p^{n_i}})\colon i\in I\}$ 
	we 
	obtain the desired conclusion for solutions to 
	$x+y=1$ with $(x,y)\in G\times G$.
	
	For part (b), we fix a 
	generator
	$\gamma$ of $\bF_q^*$. Then every solution
	$(x,y)\in (\bF_q^*)^2$ of $x+y=1$ has the form 
	$(x=\gamma^{p^k},y=(1-\gamma)^{p^k})$ for some 
	$k\in \N_0$. Let $n$ be the smallest 
	integer in $\N_0$ such that 
	$(\gamma^{p^n},(1-\gamma)^{p^n})\in G^2$
	\emph{if there exists such an $n$},
	then define
	$$\scrX':=\scrX\cup\{(\gamma^{p^n},(1-\gamma)^{p^n})\}.$$
	Otherwise, if such an $n$ does not exist, define 
	$\scrX':=\scrX$.
	\end{proof}
	
	In view of Proposition~\ref{prop:refined n} and 
	Proposition~\ref{prop:refined 2}, we ask the 
	following question:
	\begin{question}\label{q:C=0}
	Consider the equation $x_1+\ldots+x_n=1$ with
	$x_1,\ldots,x_n\in G$. Is it true that there is a finite
	set $\scrS'\subset L^*$ whose size is bounded only in terms of
	$n$, the rank, and torsion of $\sqrt[L]{G}$
	such that every non-degenerate solution $(x_1,\ldots,x_n)$
	has the form:
	$$x_k=\alpha_{k,1}^{p^{i_1}}\ldots \alpha_{k,n-1}^{p^{i_{n-1}}}\ \text{for $k=1,\ldots,n$}$$
	  for some $i_1,\ldots,i_{n-1}\in\N_0$, and 
	  $\alpha_{k,j}\in \scrS'$ for $0\leq j\leq n-1$.
	\end{question}


	\section{Proof of Theorem~\ref{thm:A} and Corollary~\ref{cor:A}}\label{sec:Proof A}
	\subsection{Notation and some preliminary results} For every finite separable extension
	$E/K$, let $\cO_{E}$ denote the integral 
	closure of $\cO$ in $E$, and let $q(E)$ be
	the cardinality of the finite field $\bar{\bF}_p\cap E$. 
	Let $M_{E}$ denote
	the discrete valuations on $E$ extending those
	in $M_K$ and normalized such that the value group of $E^*$
	is $\Z$. Let $S_E$ denote the finite subset
    of $S$ lying above $S$. We have the following:
    \begin{lemma}\label{lem:torsion rank}
		Let $E$ be a finite separable extension of $K$ and
		let $T$ be a finite subset of $M_E$ containing $S_E$.
		Then $q(E)\leq q(K)^{[E:K]}$ 
		and the rank of $\cO_{E,T}^*$ is at most 
		$|T|-1$. Consequently, the rank of $\cO_E^*$
		is at most $[E:K]|S|-1$.
    \end{lemma}  
	 \begin{proof}
		For every $\zeta\in \bar{\bF}_p\cap E$, we have that
		$[K(\zeta):K]$ divides $[E:K]$. Since the minimal polynomial
		of $\zeta$ over $K$ must have coefficients in 
		$\bF_{q(K)}=\bar{\bF}_p\cap K$, we have that 
		$[K(\zeta):K]=[\bF_{q(K)}(\zeta):\bF_{q(K)}]$.
		Hence $\zeta$ is contained in the finite field
		of degree $[E:K]$ over $\bF_{q(K)}$. This proves
		the first assertion.
		
		Let $\Div(T)$ denote the free abelian group
		generated by $T$. Consider the homomorphism $\cO_{E,T}^*\rightarrow \Div(T)$ defined by 
		$a\mapsto \sum_{v\in T} n_vv(a)v$. Since its kernel
		is exactly $(\cO_{K,T}^*)_{\tors}$ and its image is contained
		in the subgroup consisting of elements 
		whose sum of coefficients is zero, we have
		that the rank of $\cO_{E,T}^*$ is at most $|T|-1$.
		  
		Consequently, the rank of $\cO_E^*$ is at most 
		$|S_E|-1$. Since $|S_E|\leq [E:K]|S|$ (see
		\cite[pp.~164]{neu}), we get the desired conclusion.  
	 \end{proof}
	
	We will need
	the following result on unit equations in characteristic zero:
	\begin{lemma}\label{lem:3412}
	Let $A$ and $B$ be distinct non-zero integers neither of which
	is divisible by $p$. 
	Consider the equation:
	\begin{equation}\label{eq:cor of ESS}
		Ap^{X_1}-Ap^{X_2}+Bp^{X_3}-Bp^{X_4}=0\ \text{for}\ X_1,\ldots,X_4\in \N_0
	\end{equation}
	Then there exists a set $\scrD$ (depending
	on $p$, $A$, and $B$)
	of size at most $\exp(4\times 18^9)+1$ such that every solution 
	$(x_1,x_2,x_3,x_4)$ of 
	\eqref{eq:cor of ESS} satisfies $(x_3-x_4)-(x_1-x_2)\in \scrD$.
	\end{lemma}
	\begin{proof}
		Dividing by $Bp^{X_4}$, we have:
		\begin{equation}\label{eq:affine form}
			\frac{A}{B}p^{X_1-X_4}-\frac{A}{B}p^{X_2-X_4}
			+p^{X_3-X_4}=1
		\end{equation}
		with the solution 
		$\bfu=(x_1-x_4,x_2-x_4,x_3-x_4)$.	
		There are four cases:
		\begin{itemize}
			\item [(a)] No proper subsums of the left hand side of 
			\eqref{eq:affine form} vanish. 
			Proposition~\ref{prop:unit char 0}
			shows that there are at most $\exp(4\times 18^9)$
			possibilities for $\bfu$. Hence at most
			$\exp(4\times 18^9)$ possibilities
			for 
			$(x_3-x_4)-(x_1-x_4)+(x_2-x_4)=(x_3-x_4)-(x_1-x_2)$.
			
			\item [(b)] $\displaystyle p^{x_3-x_4}=1$ and $\displaystyle p^{x_1-x_4}-p^{x_2-x_4}=0$. This implies $(x_3-x_4)-(x_1-x_2)=0$.
			
			\item [(c)] $\displaystyle-\frac{A}{B}p^{x_2-x_4}=1$ and
			$\displaystyle\frac{A}{B}p^{x_1-x_4}+p^{x_3-x_4}=0$. Since $\gcd(A,p)=\gcd(B,p)=1$, we must have that 
			$A=-B$, $x_2=x_4$ and $x_1=x_3$. This gives
			$(x_3-x_4)-(x_1-x_2)=0$.
			
			\item [(d)] $\displaystyle \frac{A}{B}p^{x_1-x_4}=1$
			and 
			$\displaystyle -\frac{A}{B}p^{x_2-x_4}+p^{x_3-x_4}=0$. 
			Since $\gcd(A,p)=\gcd(B,p)=1$, we must have that 
			$A=B$. Since we assume that $A\neq B$, 
			this case cannot happen.
			\end{itemize}	
	
	The desired set $\scrD$ is obtained from the possibilities 
	of $(x_3-x_4)-(x_1-x_2)$ in Case (a)
	together with the element $0$ from Cases (b) and (c). 	
	\end{proof}
	
	\subsection{Proof of Theorem~\ref{thm:A}}
	Throughout this subsection, assume the notation in
	Theorem~\ref{thm:A} and let $L$ denote the Galois closure of $K(s)$ in $\overline{K}$. The case $d:=[K(s):K]=2$ is immediate, as follows. Assume $\cO[s]=\cO[t]$, then we can write
	$t=\alpha s+\beta$ and $s=\gamma t+\delta$ for unique
	$\alpha,\beta,\gamma,\delta\in \cO$. This implies that 
	$\alpha\gamma=1$, hence $\alpha\in \cO^*$. This proves
	Theorem~\ref{thm:A} when $d=2$ (with $t_1=s$).
	For the rest of the proof, we assume  $d\geq 3$.

	Write $q(L)=p^{\lambda}$. 
	By Lemma~\ref{lem:torsion rank}, we have:
	\begin{equation}\label{eq:q of L}
	q(L)\leq q(K)^{[L:K]}\leq q(K)^{d!},\ \text{hence}\ 
	\lambda\leq d!\log_p(q(K)).
	\end{equation}
	Let $\{\id=\sigma_1,\ldots ,\sigma_d\}$ be a 
	choice of representatives of the left cosets of $\Gal(L/K(s))$
	in
	$\Gal(L/K)$. For every element $\alpha\in K(s)$
	and for $1\leq i\leq d$, we denote
	$\alpha_{(i)}=\sigma_i(\alpha)$. In particular, 
	$s=s_{(1)},\ldots,s_{(d)}$ are all the conjugates
	of $s$ over $K$. Let $G$ be the radical in $L$ of
	the group generated by the following:
		\begin{itemize}
			\item [(i)] The group of units of 
			$\cO[s_{(i)}-s_{(j)}]$ for $1\leq i\neq j\leq d$.
			\item [(ii)] The elements $s_{(i)}-s_{(j)}$
			for $1\leq i\neq j\leq d$.   		
		\end{itemize}
	Let $r$ denote the rank of $G$. By 
	Lemma~\ref{lem:torsion rank} and \eqref{eq:q of L}, 
	we have the following:
	\begin{equation}\label{eq:tor rk of G}
		|G_{\tors}| \leq q(L)-1<q(K)^{d!}\ \text{and}\ r\leq \frac{d(d-1)}{2}(d^2|S|-1)+\frac{d(d-1)}{2}<d^4|S|.
	\end{equation}
	
	The rest of this subsection is used to
	prove the following 
	more precise version of Theorem~\ref{thm:A}:
	\begin{theorem}\label{thm:A precise}
		There are $N\leq \left(\min\{q(L),q(K)^{d^3}\}\right)^{d^3}+\left(\exp(18^{10}) p^{2r}d^8\lambda\right)^{d^3}$ elements $t_1,\ldots,t_N$
		satisfying the conditions (a) and (b) of Theorem~
		\ref{thm:A}.	
	\end{theorem}
	
	\begin{remark}\label{rem:A precise}
	   By \eqref{eq:q of L}, \eqref{eq:tor rk of G},
	   and Theorem~\ref{thm:A precise},
	   the bound in Theorem~\ref{thm:A precise}
	   is less than
	    $q(K)^{d^6}+
	   (\exp(18^{10})p^{2d^4|S|}d^{8}(d!)\log_p(q(K)))^{d^3}$ which is less than
	   $$q(K)^{d^6}+\left(\exp(18^{10})p^{3d^4|S|}\log_p q(K)
	   \right)^{d^3}.$$
	   This proves Theorem~\ref{thm:A}. Note that if we simply
	   used $q(L)$ instead of $\min\{q(L),q(K)^{d^3}\}$
	   for the bound in Theorem~\ref{thm:A precise}, then 
	   we would, a priori, have the 
	   \emph{doubly exponential}
	   expression $q(K)^{d!d^3}$
	   instead of $q(K)^{d^6}$ for 
	   the bound in Theorem~\ref{thm:A}.
	\end{remark}

	Now assume that $t$ satisfies $\cO[t]=\cO[s]$. By
	writing 
	$t=P_1(s)$ and $s=P_2(t)$ for polynomials $P_1(X),P_2(X)\in 
	\cO[X]$, we have that for $1\leq i\neq j\leq d$:  
	 \begin{equation}\label{eq:titj sisj}
	 \frac{t_{(i)}-t_{(j)}}{s_{(i)}-s_{(j)}}\in  (\cO[s_{(i)},s_{(j)}])^*,\ \text{hence}\ t_{(i)}-t_{(j)}\in G.
	 \end{equation}
	This implies that for every triple $(i,j,k)$
	of distinct elements in $\{1,\ldots,d\}$, the elements:
	$$x=(t_{(i)}-t_{(j)})/(t_{(k)}-t_{(j)})$$ 
	$$y=(t_{(k)}-t_{(i)})/(t_{(k)}-t_{(j)})$$
	give a solution to $X+Y=1$ with $x,y\in G$.
	
	By Proposition~\ref{prop:refined 2}, there exists
	a subset $\{(x_i,y_i):\ 1\leq i\leq M\}$ of $L^*$
	of size $M\leq p^{2r}$ such that each solution $(x,y)$ to
	$X+Y=1$ with $x,y\in G$ is of the form
	$(x_i^{p^j},y_i^{p^j})$ for some $i\in \{1,\ldots,M\}$
	and $j\in \N_0$. Moreover, we may assume that $x_i\notin 
	L^{\langle p\rangle}$ whenever $x_i\notin \bar{\bF}_p$.

	Let $\scrT(d)$ denote the set of all triples
	$(i,j,k)$ of distinct $i,j,k\in \{1,\ldots,d\}$ (hence $|\scrT(d)|=d(d-1)(d-2)$).
	Now for each sequence 
	$\bfm:=(m_{i,j,k})\in \{1,\ldots ,M\}^{\scrT(d)}$
	indexed by the triples $(i,j,k)\in \scrT(d)$, 
	consider the set $X_{\bfm}$ of all $t\in \cO[s]$ for which 
	$\cO[t]=\cO[s]$ and such that there exists some sequence of 
	non-negative integers $\bfa:=(a_{i,j,k})$ 
	(indexed by $(i,j,k)\in \scrT(d)$) satisfying:
 	\begin{equation}\label{eq:define aijk}
 	(t_{(i)}-t_{(j)})/(t_{(k)}-t_{(j)})=x_{m_{i,j,k}}^{p^{a_{i,j,k}}}
 	\end{equation}
 	for all $(i,j,k)\in \scrT(d)$. We let
 	$T_{\bfm}$
 	denote the set of triples
 	$(i,j,k)\in \scrT(d)$
 	such that $x_{m_{i,j,k}}\notin \bar{\bF}_p$. If 
 	$t\in X_{\bfm}$ and $(i,j,k)\in T_{\bfm}$
 	then $a_{i,j,k}$
 	is determined uniquely from \eqref{eq:define aijk}.

	The case when $T_\bfm=\emptyset$ (i.e. $x_{m_{i,j,k}}\in \bar{\bF}_p^*$ for every $(i,j,k)$) is rather easy, as follows:
	\begin{lemma}\label{lem:Tbfm empty}
		Let
		$\displaystyle X_{\alg}:=\bigcup_{\bfm:\ T_\bfm=\emptyset}X_\bfm$.
		There are at most
		$\left(\min\{q(L),q(K)^{d^3}\}\right)^{d^3}$ elements $t_1,\ldots,t_N\in X_{\alg}$
		such that every $t\in X_{\alg}$ has the form
		$t=at_i+b$ for some $1\leq i\leq N$, $a\in \cO^*$,
		and $b\in \cO$.
	\end{lemma} 	
 	\begin{proof}
	    Define the relation $\approx $ in $X_{\alg}$ as follows. 
	    Let $t,t'\in X_{\alg}$, define $t\approx t'$
	    if $t'=at+b$ for some $a\in \cO^*$
	    and $b\in \cO$. It is immediate that this is
	    an equivalence relation. It remains
	    to show that in every subset $\scrA$ of $X_{\alg}$
	    having more than 
	    $\left(\min\{q(L),q(K)^{d^3}\}\right)^{d^3}$ elements, 
	    there exist
	    two elements that are equivalent to each other.
	    
	    For $(i,j,k)\in \scrT(d)$, let $\mu_{i,j,k}=\bar{\bF}_p^*
	    \cap K(s_{(i)},s_{(j)},s_{(k)})\subseteq \bF_{q(L)}^*$.
	    Hence $|\mu_{i,j,k}|\leq \min\{q(L),q(K)^{d^3}\}$
	    by Lemma~\ref{lem:torsion rank}.
	    For $t\in X_{\alg}$, we have
	    $\displaystyle\frac{t_{(i)}-t_{(j)}}{t_{(k)}-t_{(j)}}\in \mu_{i,j,k}$
	    for every $(i,j,k)\in \scrT(d)$. Since 
	    $|\scrA| > |\prod_{(i,j,k)} \mu_{i,j,k}|$,
	    there exist
	    $t,t'\in \scrA$ such that 
	    $\displaystyle\frac{t_{(i)}-t_{(j)}}{t_{(k)}-t_{(j)}}=\frac{t'_{(i)}-
	    t'_{(j)}}{t'_{(k)}-t'_{(j)}}$
	    for every $(i,j,k)\in \scrT(d)$. Equivalently,
	    the element
	    $\displaystyle a:=\frac{t'_{(i)}-t'_{(j)}}{t_{(i)}-t_{(j)}}$
	    is independent of distinct $i,j\in 
	    \{1,\ldots,d\}$, and belongs to $\cO[s_{(i)},s_{(j)}]^*$
	    by \eqref{eq:titj sisj}. Hence $a\in \cO^*$
	    since it is non-zero, invariant under $\Gal(L/K)$, and 
	    integral
	    over $\cO$. 
	    
	    Now the element
	    $b:=t'_{(i)}-at_{(i)}$
	    is independent of $i\in \{1,\ldots,d\}$. Hence $b\in\cO$
	    since it is invariant under $\Gal(L/K)$ and integral 
	    over $\cO$. This finishes the proof.
 	\end{proof}

 	It remains to investigate the case $T_\bfm\neq \emptyset$. We
 	have the following useful observation:
 	
 	\begin{lemma}\label{lem:useful A} 
 	Let $t'\in X_{\bfm}$ and write:
 		$$(t'_{(i)}-t'_{(j)})/(t'_{(k)}-t'_{(j)})=x_{m_{i,j,k}}^{p^{b_{i,j,k}}}$$
 	for a sequence of non-negative integers $(b_{i,j,k})$. We have:
 	\begin{itemize}
		\item [(a)] If $(i,j,k)\in T_\bfm$
		then $x_{m_{i,j,k}}=x_{m_{k,j,i}}^{-1}$
		and $b_{m_{i,j,k}}=b_{m_{k,j,i}}$.
		\item [(b)] Let $\sigma\in \Gal(L/K)$ and 
		$(i,j,k)\in T_\bfm$. Let $(i_1,j_1,k_1)\in\scrT(d)$
		be such that $\sigma(t_{(i)})=t_{(i_1)}$, $
		\sigma(t_{(j)})=t_{(j_1)}$, $\sigma(t_{(k)})=t_{(k_1)}$.
		Then $\sigma(x_{m_{i,j,k}})=x_{m_{i_1,j_1,k_1}}$ and
		$b_{i,j,k}=b_{i_1,j_1,k_1}$.
 	\end{itemize}
 	\end{lemma}
 	\begin{proof}
 		For part (a), note the identity:
 		$$x_{m_{i,j,k}}^{p^{b_{i,j,k}}}=
 		\left(x_{m_{k,j,i}}^{p^{b_{k,j,i}}}\right)^{-1}.$$
 		This implies that $x_{m_{k,j,i}}$
 		is not in $\bar{\bF}_p$ either; hence both 
 		$x_{m_{i,j,k}}$ and $x_{m_{j,k,i}}$
 		are
 		not in $L^{\langle p \rangle}$ by our choice of
 		the set $\{x_i:\ 1\leq i\leq M\}$. 
 		If $b_{i,j,k}< b_{k,j,i}$ (respectively 
 		$b_{i,j,k}>b_{k,j,i}$)
 		then we would have $x_{m_{i,j,k}}\in L^{\langle p \rangle}$
 		(respectively $x_{m_{k,j,i}}\in L^{\langle p \rangle}$),
 		contradiction.
 		Therefore
 		$b_{i,j,k}=b_{k,j,i}$ and 
 		$x_{m_{i,j,k}}=x_{m_{k,j,i}}^{-1}$.
 		
 		For part (b), we argue similarly by using the identity:
 		$$\sigma(x_{m_{i,j,k}})^{p^{b_{i,j,k}}}=
 		x_{m_{i_1,j_1,k_1}}^{p^{b_{i_1,j_1,k_1}}}.$$
 	\end{proof}

 	We need the following technical result:
 	\begin{proposition}\label{prop:first technical}
		 Let
		 $\bfm:=(m_{i,j,k})\in \{1,\ldots ,M\}^{\scrT(d)}$ and $t\in X_{\bfm}$ with  
	$$(t_{(i)}-t_{(j)})/(t_{(k)}-t_{(j)})=x_{m_{i,j,k}}
	^{p^{a_{i,j,k}}}$$ 
	for every $(i,j,k)\in \scrT(d)$
	for some sequence of non-negative integers 
	$\bfa=(a_{i,j,k})$.
	There exists a set $J\subseteq \Z$ (possibly depending on $\bfm$, $t$, and $\bfa$) such that the following holds:
	\begin{itemize}
		\item [(a)] $|J|\leq d^4(1+\exp(4\times 18^9))$
		
		\item [(b)] For every $t'\in X_{\bfm}$, for $(i,j,k)\in T_\bfm$, let $b_{m_{i,j,k}}$
		be the unique non-negative integer such that
		$$(t'_{(i)}-t'_{(j)})/(t'_{(k)}-t'_{(j)})=x_{m_{i,j,k}}^{p^{b_{i,j,k}}}.$$	
		For any four distinct elements 
		$i,j,k,\ell\in\{1,\ldots,M\}$
		such that $(i,j,k)\in T_\bfm$ and $(i,j,\ell)\in T_\bfm$,
		we have $(b_{i,j,k}-a_{i,j,k})-(b_{i,j,\ell}-a_{i,j,\ell})$
		is in $J$.
	\end{itemize}
	\end{proposition}
					
	\begin{proof}
	Since there are less than $d^4$ quadruples of
	distinct elements $(i,j,k,\ell)$,
	it suffices to fix any four distinct elements $i,j,k,\ell$
	such that $(i,j,k)$ and $(i,j,\ell)$ are in $T_\bfm$ and
	prove that there are at most $\exp(4\times 18^9)+1$ possibilities
	(independent of $t'$)
	for 
	$\Delta:=(b_{i,j,k}-a_{i,j,k})-(b_{i,j,\ell}-a_{i,j,\ell})$. 
	Part (a) of Lemma~\ref{lem:useful A} 
	gives $x_{m_{i,j,k}}=x_{m_{k,j,i}}^{-1}$, 
	$x_{m_{i,j,\ell}}=x_{m_{\ell,j,i}}^{-1}$,
	$a_{m_{i,j,k}}=a_{m_{k,j,i}}$, $b_{m_{i,j,k}}=b_{m_{k,j,i}}$,
	$a_{m_{i,j,\ell}}=a_{m_{\ell,j,i}}$, and 
	$b_{m_{i,j,\ell}}=b_{m_{\ell,j,i}}$. 
	These identities will be used many times in the
	proof. 
	
	Observe that 
$$1=\frac{(t_{(i)}-t_{(j)})}{(t_{(k)}-t_{(j)})} \cdot 
\frac{(t_{(k)}-t_{(j)})}{(t_{(\ell)}-t_{(j)})}  \cdot \frac{(t_{(\ell)}-t_{(j)})}{(t_{(i)}-t_{(j)})}.$$
 This relation gives that
 \begin{equation}
 \label{eq: aijk}
 1=x_{m_{i,j,k}}^{p^{a_{i,j,k}}} x_{m_{k,j,\ell}}^{p^{a_{k,j,\ell}}} x_{m_{\ell, j,i}}^{p^{a_{\ell,j,i}}}.
 \end{equation}
 
 Using the similar expression involving $t'_{(i)},t'_{(j)},t'_{(k)},t'_{(\ell)}$ gives
 \begin{equation}
 \label{eq: bijk}
 1=x_{m_{i,j,k}}^{p^{b_{i,j,k}}} x_{m_{k,j,\ell}}^{p^{b_{k,j,\ell}}} x_{m_{\ell, j,i}}^{p^{b_{\ell,j,i}}}.
 \end{equation}
 
 Raising both sides of (\ref{eq: aijk}) to the power $p^{b_{k,j,\ell}}$ and raising both sides of (\ref{eq: bijk}) to the power $p^{a_{k,j,\ell}}$ and then dividing yields
\begin{equation}
\label{eq: red}
1=x_{m_{i,j,k}}^{p^{a_{i,j,k}+b_{k,j,\ell}} - p^{b_{i,j,k}+a_{k,j,\ell}}} x_{m_{\ell,j,i}}^{p^{a_{\ell,j,i}+b_{k,j,\ell} }- p^{b_{\ell,j,i}+a_{k,j,\ell}}}.
\end{equation}
	
	We now consider two cases:
	
		\textbf{Case 1:} $x_{m_{i,j,k}}$ and $x_{m_{\ell,j,i}}$
		generate a rank two abelian subgroup of $L^*$.	We claim
		that $b_{i,j,k}-a_{i,j,k}=b_{i,j,\ell}-a_{i,j,\ell}$.

By the assumption in this case and \eqref{eq: red}, we must have 
$$a_{i,j,k}+b_{k,j,\ell} = b_{i,j,k}+a_{k,j,\ell}\ 
\text{and}\ a_{\ell,j,i}+b_{k,j,\ell}=b_{\ell,j,i}+a_{k,j,\ell}.$$  This implies:
$$b_{i,j,k}-a_{i,j,k}=b_{k,j,\ell}-a_{k,j,\ell}=b_{\ell,j,i}-a_{\ell,j,i}$$
which proves the desired claim. We now simply choose
$J(i,j,k,\ell)=\{0\}$ in this case.

	\textbf{Case 2:}  $x_{m_{i,j,k}}$ and $x_{m_{\ell,j,i}}$ generate a rank one abelian subgroup of $L^*$. Let $\Gamma$
	be the radical in $L$ of this rank one subgroup and
	let $u\in L^*$ be a generator of the infinite cyclic
	group $\Gamma/\Gamma_{\tors}$.  
	Hence $u\notin\bar{\bF}_p$ and
	there exist non-zero integers $A$ and $B$ such that
	$x_{m_{i,j,k}}u^{-A}$ and $x_{m_{\ell,j,i}}u^{-B}$ are both 
	in $\bar{\bF}_p^*$. Due to our choice
	that $x_{m_{i,j,k}}\notin L^{\langle p \rangle}$
	and $x_{m_{\ell,j,i}}\notin L^{\langle p \rangle}$,
	neither $A$ nor $B$ is divisible by $p$.
	Now (\ref{eq: red}) gives:
	\begin{equation}\label{eq:AB homogeneous}
	A(p^{a_{i,j,k}+b_{k,j,\ell}} - p^{b_{i,j,k}+a_{k,j,\ell}}) + 		B(p^{a_{\ell,j,i}+b_{k,j,\ell}} - p^{b_{\ell,j,i}+a_{k,j,\ell}})=0.
	\end{equation}
	
	We now have two smaller cases:
	
	\textbf{Case 2.1:} consider the case $A\neq B$.
	Lemma~\ref{lem:3412} gives that there exists 
	a set $\scrD$ (depending only on $p$, $A$, and $B$) of size at most 
	$\exp(4\times 18^9)+1$ such that
	$$(a_{\ell,j,i}+b_{k,j,\ell}-b_{\ell,j,i}-a_{k,j,\ell})-
	(a_{i,j,k}+b_{k,j,\ell}-b_{i,j,k}-a_{k,j,\ell})=\Delta$$
	  belongs to $\scrD$. We choose $J(i,j,k,\ell)=\scrD$
	  in this case.

	\textbf{Case 2.2:} consider the case $A=B$. We have
	that:
		\begin{equation}\label{eq:A=B;1st}
		\frac{x_{m_{i,j,k}}}{x_{m_{\ell,j,i}}}\in u^{A-B}\bar{\bF}_p^*=\bar{\bF}_p^*
		\end{equation}
	From \eqref{eq: aijk} and \eqref{eq:A=B;1st}, we
	have:
	\begin{equation}\label{eq:A=B;2nd}
		x_{m_{\ell,j,i}}^{p^{a_{i,j,k}}+p^{a_{\ell,j,i}}}x_{m_{k,j,\ell}}^{p^{a_{k,j,\ell}}}\in \bar{\bF}_p^*	
	\end{equation}
	This implies that $x_{m_{k,j,\ell}}\in \Gamma$
	and $x_{m_{k,j,\ell}}\notin \bar{\bF}_p^*$. Hence there
	is a non-zero integer $C$ (not divisible by $p$)
	such that $x_{m_{k,j,\ell}}u^{-C}\in\bar{\bF}_p^*$. And 
	\eqref{eq:A=B;2nd} yields:
	\begin{equation}\label{eq:A=B unit for a}
	B(p^{a_{i,j,k}}+p^{a_{\ell,j,i}})+Cp^{a_{k,j,\ell}}=0
	\end{equation}
	By similar arguments for $t'$ using \eqref{eq: bijk}, we have:
	\begin{equation}\label{eq:A=B unit for b}
	B(p^{b_{i,j,k}}+p^{b_{\ell,j,i}})+Cp^{b_{k,j,\ell}}=0
	\end{equation}
	
	By \eqref{eq:A=B unit for a} and \eqref{eq:A=B unit for b},
	we have that $$(p^{a_{i,j,k}-a_{\ell,j,i}},p^{a_{k,j,\ell}-
	a_{\ell,j,i}})\qquad {\rm and}\qquad (p^{b_{i,j,k}-b_{\ell,j,i}},p^{b_{k,j,\ell}-
	b_{\ell,j,i}})$$
	are solutions of the unit equation $-X-\frac{C}{B}Y=1$.
	By Proposition~\ref{prop:unit char 0},
	there are at most $\exp(3\times 12^6)$
	possibilities (depending only on $p$, $B$, and $C$)
	for each of $a_{i,j,k}-a_{\ell,j,i}$
	and $b_{i,j,k}-b_{\ell,j,i}$. Hence
	there are at most $\exp(6\times 12^6)$
	possibilities for $\Delta$. We choose
	$J(i,j,k,\ell)$ to be the set of such possibilities.

	In any case, we have that $J(i,j,k,\ell)$ does not
	depend on $t'$ and has at most
	$\exp(4\times 18^9)+1$ elements. This
	finishes the proof.		
	\end{proof}	
 	
	The next technical result is the key step towards the proof of Theorem~\ref{thm:A precise}. Recall that $q(L)=p^\lambda$.
	
	\begin{proposition}\label{prop:key A}
	Let $\bfm:=(m_{i,j,k})\in \{1,\ldots ,M\}^{\scrT(d)}$ such
	that $T_\bfm\neq\emptyset$ and $X_\bfm\neq\emptyset$. Fix 
	$(i_0,j_0,k_0)\in T_\bfm$ and $t\in X_{\bfm}$ with 
	$$(t_{(i)}-t_{(j)})/(t_{(k)}-t_{(j)})=x_{m_{i,j,k}}
	^{p^{a_{i,j,k}}}$$ 
	for every $(i,j,k)\in \scrT(d)$
	for some sequence of non-negative integers $
	\bfa=(a_{i,j,k})$. 
	There exists a set $I\subseteq \Z$ (possibly depending on $\bfm$, $t$, $\bfa$, and $(i_0,j_0,k_0)$) such that the following hold:
	\begin{itemize}
		\item [(a)] $|I|\leq d^3\lambda+2d^8\exp(8\times 18^9)$.
		
		\item [(b)] For every $t'\in X_{\bfm}$, there exists a sequence $(b_{i,j,k})$ of non-negative integers satisfying 
		the two conditions:
		\begin{itemize}
			\item [(i)]
$(t'_{(i)}-t'_{(j)})/(t'_{(k)}-t'_{(j)})=x_{m_{i,j,k}}^{p^{b_{i,j,k}}}$ for every  $(i,j,k)\in\scrT(d)$. 
			\item [(ii)] $(b_{i,j,k}-a_{i,j,k})-(b_{i_0,j_0,k_0}-a_{i_0,j_0,k_0})\in I$ for any $(i,j,k)\in \scrT(d)$.	
		\end{itemize}	
	\end{itemize}
	\end{proposition} 
 
 	\begin{proof}
 	Let $t'\in X_{\bfm}$. By the definition of 
 	$X_{\bfm}$, 
 	we can choose a sequence $(c_{j,i,k})$ 
 	of non-negative integers
 	such that 
	$$(t'_{(i)}-t'_{(j)})/(t'_{(k)}-t_{(j)})=x_{m_{i,j,k}}^{p^{c_{i,j,k}}}$$ 
	for every $(i,j,k)\in \scrT(d)$.  The goal
	is to modify the sequence $(c_{i,j,k})$
	into the sequence $(b_{i,j,k})$ such that
	for any $(i,j,k)\in \scrT(d)$, 
	the number
	$(b_{i,j,k}-a_{i,j,k})-(b_{i_0,j_0,k_0}-a_{i_0,j_0,k_0})$
	lies in a set $I$ independent of $t'$
	whose size is at most the given bound.
    

	Recall
	that for $(i,j,k)\in T_\bfm$, the value of $b_{i,j,k}$
	is uniquely determined and is equal to $c_{i,j,k}$. For
	$(i,j,k)\in\scrT(d)\setminus T_\bfm$, 
	we have that 
	$x_{m_{i,j,k}}\in \bar{\bF}_p^*\cap L^*=\bF_{q(L)}^*$,
	hence $x_{m_{i,j,k}}^{p^\lambda}=x_{m_{i,j,k}}$. This
	allows us to replace $c_{i,j,k}$ by
	$c_{i,j,k}+\omega \lambda$
	for any integer $\omega$. We now define
	$b_{i,j,k}$ as follows.
	Let $\gamma_{i,j,k}$
	be the smallest non-negative integer satisfying:
	$$-a_{i_0,j_0,k_0}+a_{i,j,k}+\gamma_{i,j,k}\lambda\geq 0.$$
	We now let $b_{i,j,k}$
	to be of the form $c_{i,j,k}+\omega\lambda$
	for some integer $\omega$ such that:
	$$b_{i_0,j_0,k_0}-a_{i_0,j_0,k_0}+a_{i,j,k}+\gamma_{i,j,k}\lambda\leq b_{i,j,k}<b_{i_0,j_0,k_0}-a_{i_0,j_0,k_0}+a_{i,j,k}+(\gamma_{i,j,k}+1)\lambda.$$
	
	For $(i,j,k)\in \scrT(d)$, denote
	$\Delta_{i,j,k}:=b_{i,j,k}-a_{i,j,k}$. The bottom line
	of the above definition of $\gamma_{i,j,k}$ and $b_{i,j,k}$
	is that the following properties hold for every $(i,j,k)\in\scrT(d)\setminus T_\bfm$:
	\begin{equation}\label{eq:gamma and t'}
		\text{$\gamma_{i,j,k}$ does not depend on $t'$;}
	\end{equation}
	\begin{equation}\label{eq:bijk non-negative}
		\text{$b_{i,j,k}$ is non-negative;}
	\end{equation}
	\begin{equation}\label{eq:gammadelta}
	\gamma_{i,j,k}\lambda\leq \Delta_{i,j,k}-\Delta_{i_0,j_0,k_0}
	<(\gamma_{i,j,k}+1)\lambda.
	\end{equation}
		

	Define $\displaystyle I_1=\bigcup_{(i,j,k)\in \scrT(d)\setminus T_\bfm} 
	\{\alpha\in\Z: \gamma_{i,j,k}\lambda\leq \alpha<(\gamma_{i,j,k}+1)\lambda\}$.
	Let $J$ be the set of size at most $d^4(\exp(4\times 18^9)+1)$
	as in Proposition~\ref{prop:first technical} and define
	$I_2:=\{\alpha+\beta:\ \alpha,\beta\in J\}$.
	We now define:
	$$I:=\{0\}\cup I_1\cup I_2$$
	which gives that:
	$$|I|\leq 1+d(d-1)(d-2)\lambda+d^8(\exp(4\times 18^9)+1)^2
	< d^3\lambda+2d^8\exp(8\times 18^9).$$
	We need to prove:
	\begin{equation}\label{eq:scrW}
	\Delta_{i,j,k}-\Delta_{i_0,j_0,k_0}\in I\ 
	\text{for every $(i,j,k)\in \scrT(d)$.} 
	\end{equation}
	
	This holds trivially when $(i,j,k)=(i_0,j_0,k_0)$. 
	By \eqref{eq:gammadelta}, we have that \eqref{eq:scrW}
	holds when $(i,j,k)\in \scrT(d)\setminus T_\bfm$. 
	It remains to consider
	$(i,j,k)\in T_\bfm$ and $(i,j,k)\neq (i_0,j_0,k_0)$.
 	Since $\Gal(L/K)$
	acts transitively on the set 
	$\{t'_{(1)},\ldots,t'_{(d)}\}$, we can find 
	$\sigma\in \Gal(L/K)$ such that $\sigma(t'_{(j)})=t'_{(j_0)}$
	and let $i_1,k_1\in\{1,\ldots,d\}$ such that 
	$\sigma(t'_{(i)})=t'_{(i_1)}$
	and $\sigma(t'_{(k)})=t'_{(k_1)}$. By Lemma~\ref{lem:useful A},
	we have:
	\begin{equation}\label{eq:i1k1}
	\Delta_{i,j,k}-\Delta_{i_0,j_0,k_0}=\Delta_{i_1,j_0,k_1}-\Delta_{i_0,j_0,k_0}.
	\end{equation}
	
	If $(i_1,j_0,k_0)\in T_\bfm$ then Proposition~\ref{prop:first technical} and part (a) of Lemma~\ref{lem:useful A} give:
		$$\Delta_{i_1,j_0,k_1}-\Delta_{i_0,j_0,k_0}=(\Delta_{i_1,j_0,k_1}-\Delta_{i_1,j_0,k_0})+
		(\Delta_{i_1,j_0,k_0}-\Delta_{i_0,j_0,k_0})\in I_2.$$
	The case when $(i_0,j_0,k_1)\in T_\bfm$ is handled
	similarly.	The only case left is that both
	$x_{m_{i_1,j_0,k_0}}$ and $x_{m_{i_0,j_0,k_1}}$
	are in $\bar{\bF}_p^*$.	In this case, we have:
	$$\alpha:=\frac{t_{(i_1)}-t_{(j_0)}}{t_{(k_0)}-t_{(j_0)}};\ 
	\beta:=\frac{t_{(i_0)}-t_{(j_0)}}{t_{(k_1)}-t_{(j_0)}}$$
	$$\alpha':=\frac{t'_{(i_1)}-t'_{(j_0)}}{t'_{(k_0)}-t'_{(j_0)}};\ 
	\beta':=\frac{t'_{(i_0)}-t'_{(j_0)}}{t'_{(k_1)}-t'_{(j_0)}}$$ 
	are all contained in $\bar{\bF}_p^*$. Since
	$$\frac{t_{(i_1)}-t_{(j_0)}}{t_{(k_1)}-t_{(j_0)}}
	=\alpha\beta\frac{t_{(k_0)}-t_{(j_0)}}{t_{(i_0)}-t_{(j_0)}}$$
	$$\frac{t'_{(i_1)}-t'_{(j_0)}}{t'_{(k_1)}-t'_{(j_0)}}
	=\alpha'\beta'\frac{t'_{(k_0)}-t'_{(j_0)}}{t'_{(i_0)}-t'_{(j_0)}}$$
	we have
	\begin{equation}\label{eq:alphabeta}
		x_{m_{i_1,j_0,k_1}}^{p^{a_{i_1,j_0,k_1}}}=\alpha\beta
		x_{m_{k_0,j_0,i_0}}^{p^{a_{k_0,j_0,i_0}}}
	\end{equation}
	
	\begin{equation}\label{eq:alpha'beta'}
		x_{m_{i_1,j_0,k_1}}^{p^{b_{i_1,j_0,k_1}}}=\alpha'\beta'
		x_{m_{k_0,j_0,i_0}}^{p^{b_{k_0,j_0,i_0}}}.
	\end{equation}
	
	Hence the radical in $L$ of the group generated by 
	$x_{m_{i_1,j_0,k_1}}$ and $x_{m_{k_0,j_0,i_0}}$
	has rank one. As in the proof of Proposition~\ref{prop:first technical}, there exist $u\in L^*\setminus\bar{\bF}_p^*$
	and non-zero integers $A$, $B$ such that
	$x_{m_{i_1,j_0,k_1}}u^{-A}$ and $x_{m_{k_0,j_0,i_0}}u^{-B}$
	are in $\bar{\bF}_p^*$. Together with equations \eqref{eq:alphabeta} and \eqref{eq:alpha'beta'}, we have
	$Ap^{a_{i_1,j_0,k_1}}=Bp^{a_{k_0,j_0,i_0}}$
	and $Ap^{b_{i_1,j_0,k_1}}=Bp^{b_{k_0,j_0,i_0}}$. This and 
	part (a) of Lemma~\ref{lem:useful A} give 
	$\Delta_{i_1,j_0,k_1}-\Delta_{i_0,j_0,k_0}=0\in I$. This
	finishes the proof.
	\end{proof}

	Let $\bfm=(m_{i,j,k})\in \{1,\ldots,M\}^{\scrT(d)}$ such that
	$T_\bfm\neq \emptyset$ and $X_\bfm\neq \emptyset$.
	Fix $(i_0,j_0,k_0)\in T_\bfm$ and $t\in X_\bfm$ with:
	$$\frac{t_{(i)}-t_{(j)}}{t_{(k)}-t_{(j)}}=x_{m_{i,j,k}}^{p^{a_{i,j,k}}}$$
	for every $(i,j,k)\in \scrT(d)$
	for some sequence of non-negative integers $(a_{i,j,k})$. Let
	$I_{\bfm}$ denote the resulting set as in the conclusion 
	of Proposition~\ref{prop:key A}.
	
	For every
	$\bfD:=(D_{i,j,k})\in I_{\bfm}^{\scrT(d)}$, define
	$X_{\bfm,\bfD}$ to be the set of $t'\in X_\bfm$
	such that there exists an integer $C$ satisfying
	the following:
	\begin{equation}\label{eq:t' and C}
		\frac{t'_{(i)}-t'_{(j)}}{t'_{(k)}-t'_{(j)}}=x_{m_{i,j,k}}^{p^{C+D_{i,j,k}+a_{i,j,k}}}\ 
		\text{for every $(i,j,k)\in\scrT(d)$.}
	\end{equation}
	We have:
	\begin{lemma}\label{lem:union XmD}
	Notation as above, we have $\displaystyle X_\bfm=\bigcup_{\bfD\in I_{\bfm}^{\scrT(d)}} X_{\bfm,\bfD}$.
	\end{lemma}
	\begin{proof}
	    Given $t'\in X_\bfm$, by
		Proposition~\ref{prop:key A}, 
		there is a sequence $(b_{i,j,k})$ such that
		the following holds:
		\begin{itemize}
			\item [(i)] $\displaystyle\frac{t'_{(i)}-t'_{(j)}}{t'_{(k)}-t'_{(j)}}=x_{m_{i,j,k}}^{p^{b_{i,j,k}}}$ for every
			$(i,j,k)\in \scrT(d)$.
			\item [(ii)] $D_{i,j,k}:=(b_{i,j,k}-a_{i,j,k})-(b_{i_0,j_0,k_0}-a_{i_0,j_0,k_0})\in I_{\bfm}$ for every $(i,j,k)\in \scrT(d)$.  
		\end{itemize}
		Let $\bfD=(D_{i,j,k})$, we have
		$t'\in X_{\bfm,\bfD}$ (by taking $C:=b_{i_0,j_0,k_0}-a_{i_0,j_0,k_0}$).
	\end{proof}
	
	For every $z\in X_{\bfm,\bfD}$, define $e(z)$ to be the 
	largest non-negative integer $n$ such that
	$$\frac{z_{(i_0)}-z_{(j_0)}}{z_{(k_0)}-z_{(j_0)}}\in L^{\langle p^n \rangle }:=\{\alpha^{p^n}: \alpha\in L\}.$$
	Such an $n$ exists since $\displaystyle\frac{z_{(i_0)}-z_{(j_0)}}{z_{(k_0)}-z_{(j_0)}}$ is a power of $x_{m_{i_0,j_0,k_0}}$
	which is not algebraic over $\bF_p$.
	
	If $X_{\bfm,\bfD}\neq\emptyset$, define $z=z(\bfm,\bfD)$
	to be an element of $X_{\bfm,\bfD}$ such that $e(z)$
	is minimal. Let $t'$ be any element in $X_{\bfm,\bfD}$.
	By the definition of $X_{\bfm,\bfD}$, there
	is an integer $C$ such that:
	$$\frac{t'_{(i)}-t'_{(j)}}{t'_{(k)}-t'_{(j)}}=\left(\frac{z_{(i)}-z_{(j)}}{z_{(k)}-z_{(j)}}\right)^{p^C}
	\ \text{for every $(i,j,k)\in \scrT(d)$.}$$
	By the minimality of  $e(z)$, we must have that $C\geq 0$. 
	As in the proof of Lemma~\ref{lem:Tbfm empty}, we have that
	$$a:=\frac{t'_{(i)}-t'_{(j)}}{(z_{(i)}-z_{(j)})^{p^C}}$$
	is non-zero and independent of distinct 
	$i,j\in \{1,\ldots,d\}$. Hence
	$a\in K^*$ since it is invariant under $\Gal(L/K)$. The element
	$b:=t'_{(i)}-az_{(i)}^{p^C}$ is
	independent of $i\in \{1,\ldots,d\}$. So it is invariant
	under $\Gal(L/K)$ and, hence, is in $K$. Therefore
	$t'=az^{q}+b$ with $q=p^C$, $a\in K^*$, and $b\in K$. Since $\cO[z]=\cO[t']=\cO[s]$, we have that $\discr_K(z)$, $\discr_K(t')$,
	and $D:=\discr_K(s)$ differ (multiplicatively) by 
	a unit. 
	This implies $\displaystyle\frac{a^{d(d-1)}}{D^{1-q}}\in\cO^*$.
	
	We now finish the proof of Theorem~\ref{thm:A precise}, 
	as follows. The $t_i$'s  in the conclusion
	of Theorem~\ref{thm:A precise} could be taken
	to be the $t_i$'s in the conclusion of
	Lemma~\ref{lem:Tbfm empty} together with
	the elements $z(\bfm,\bfD)$ from the above discussion.
	The number of such elements is at most:
	$$\left(\min\{q(L),q(K)^{d^3}\}\right)^{d^3}+ M^{d(d-1)(d-2)}\left(d^3\lambda+2d^8\exp(8\times 18^9)\right)^{d(d-1)(d-2)}$$
	which is less than:
	$$\left(\min\{q(L),q(K)^{d^3}\}\right)^{d^3}+\left(\exp(18^{10})p^{2r}d^8\lambda\right)^{d^3}.$$
	
	\subsection{An example}\label{subsec:example}
	For the sake of completeness, we construct 
	an example to show that it is not always
	possible to have $t=at_i^q+b$
	as in Theorem~\ref{thm:A} with the
	further restriction that $b$ is in $\cO$.
	
	Let $\cO=\bF_2[x]$, $K=\bF_2(x)$, and let $\eta\in 
	\cO\setminus\bF_2$ be a non-constant polynomial 
	such that the following properties hold:
	\begin{itemize}
		\item [(i)] $x$ does not divide
		$\eta$ (as polynomials in $\bF_2[x]$);
		
		\item [(ii)] the polynomial
		$P(Y):=Y^4+x^4Y^2+x^3Y+\eta$
		is irreducible over $K$.  
	\end{itemize}
	It is easy to check that, for instance, $\eta=x+1$ satisfies
	the above conditions. In fact, there are infinitely many such
	$\eta$'s. 
	
	We now let $s$ be a root of $P(Y)$. Since
	$s$ is separable over $K$, we have that
	$s^{4^m}\neq K$ for every $m\in\N$. 
%
	The sequence $\{\eta_m\}_{m\in\N}$
	of elements of $\cO$
	is
	defined recursively as follows:
	$$\eta_1=\eta;\ \eta_{m+1}=\eta^{4^m}+x^{3\cdot 4^m}\eta_m+x^{4^{m+1}}\eta_m^2\ \text{for $m\in \N$}.$$
	For $m\in\N$, define $\displaystyle z_m:=\frac{s^{4^m}+\eta_m}{x^{4^m-1}}$. Since $\discr_K(s)=x^{12}$, we have that
	$\discr_K(s)=\discr_K(z_m)$ for 
	every $m\in\N$.
	We have:
	
	\begin{lemma}\label{lem:OsOzm}
		$\cO[s]=\cO[z_m]$ for every $m\in\N$.
	\end{lemma}
	\begin{proof}
		Since $\discr_K(s)=\discr_K(z_m)$, it suffices
		to show that $z_m\in \cO[s]$ for every $m\in\N$.
		From $P(s)=0$, we have $\displaystyle z_1=\frac{s^4+\eta}{x^3}=xs^2+s \in\cO[s]$. We complete the proof by proving
		that $z_{m+1}\in\cO[z_m]$ for every $m\in\N$.
		
		From $P(s)^{4^m}=0$, we have:
		$$\displaystyle\frac{s^{4^{m+1}}+\eta^{4^m}}{x^{4^{m+1}-1}}=xs^{2\cdot 4^m}+\frac{s^{4^m}}{x^{4^m-1}}.$$
		Adding $\displaystyle\frac{1}{x^{4^{m+1}-1}}
		\left(x^{3\cdot 4^m}\eta_m+x^{4^{m+1}}\eta_m^2\right)$
		to both sides and using the recurrence relation
		defining $\{\eta_m\}$, we have:
		$$z_{m+1}=x\left(s^{4^m}+\eta_m\right)^2+\frac{s^{4^m}+\eta_m}{x^{4^m-1}}\in\cO[z_m].$$
	\end{proof}
	
	Let $v$ denote the discrete valuation on $\cO$ such that 
	$v(x)=1$. We have:
	\begin{lemma}\label{lem:eta-eta^4}
		$v(\eta_{m+1}-\eta_m^4)=4^{m+1}-4$ for every $m\in\N$.
	\end{lemma}
	\begin{proof}
	From $v(\eta_1)=0$ and $\eta_2-\eta_1^4=x^{12}\eta_1+x^{16}\eta_1^2$, the lemma holds when $m=1$. Using
	$$\eta_{m+1}-\eta_m^4= x^{3\cdot 4^m}(\eta_m-\eta_{m-1}^4)
	+x^{4^{m+1}}(\eta_m-\eta_{m-1}^4)^2$$
	thanks to the recursive formula for $\{\eta_m\}$
	and by induction, the lemma holds for every $m\in\N$.
	\end{proof}
	
	The next result shows that we have the desired example:
	\begin{proposition}
		There does not exist a finite set $\{t_1,\ldots,t_N\}$
		satisfying the following conditions:
		\begin{itemize}
			\item [(a)] $\cO[t_i]=\cO[s]$ for $1\leq i\leq N$;
			\item [(b)] for every $m\in\N$, there exist
			$i\in\{1,\ldots,N\}$, a power $q$ of $p$, elements
			$a\in K$ and $b\in\cO$ such that
			$z_m=at_i^q+b$.
		\end{itemize}	
	\end{proposition}
	\begin{proof}
	Assume there exists such a finite set. Then there are 
	$t$ such that $\cO[t]=\cO[s]$ and
	an
	infinite subset $\cM\subseteq \N$ such that for every
	$m\in\cM$, there are a power $q_m$ of $p$, $a_m\in K$, and
	$b_m\in\cO$ such that 
	$z_m=a_mt^{q_m}+b_m$. After replacing $\cM$ by 
	an infinite subset if necessary, we may assume that 
	$q_m\leq q_n$ for $m,n\in\cM$ with $m<n$.
	
	Let $j$ be the smallest element of $\cM$, for every
	$m\in\cM$, we can write:
	\begin{equation}\label{eq:zmzj}
	z_m=\frac{a_m}{a_j^{q_m/q_j}}\left(a_jt^{q_j}+b_j\right)^{q_m/q_j}+b_m-\frac{a_m}{a_j^{q_m/q_j}}b_j^{q_m/q_j}=c_m z_j^{q_m/q_j}+d_m
	\end{equation}
	where $c_m=\displaystyle\frac{a_m}{a_j^{q_m/q_j}}$
	and $d_m=\displaystyle b_m-\frac{a_m}{a_j^{q_m/q_j}}b_j^{q_m/q_j}$. Recall that $v$ is the discrete valuation on $\cO$ with $v(x)=1$. 
	By comparing discriminant, we have $a_m^{12}=\discr_K(s)^{1-q_m}$ and $a_j^{12}=\discr_K(s)^{1-q_j}$. Therefore:
	\begin{equation}\label{eq:1-qm/qj}
	\left(\frac{a_m}{a_j^{q_m/q_j}}\right)^{12}=\discr_K(s)^{1-q_m/q_j}=x^{12(1-q_m/q_j)}.
	\end{equation}
	Hence $\displaystyle v\left(\frac{a_m}{a_j^{q_m/q_j}}\right)=1-\frac{q_m}{q_j}$. Since $b_m\in\cO$ for every $m\in\cM$, we have:
	\begin{equation}\label{eq:vdm}
		v(d_m)\geq 1-\frac{q_m}{q_j}\ \text{for $m\in\cM$}.
	\end{equation}
	 
	We can rewrite \eqref{eq:zmzj} as:
	\begin{equation}\label{eq:rewrite zmzj}
	\frac{s^{4^m}+\eta_m}{x^{4^m-1}}=c_m
	\left(\frac{s^{4^j}+\eta_j}{x^{4^j-1}}\right)^{q_m/q_j}+d_m.
	\end{equation} 
	\textbf{Claim:} $\displaystyle \frac{q_m}{q_j}=4^{m-j}$
	and $\displaystyle \frac{c_m}{x^{(4^j-1)q_m/q_j}}=\frac{1}{x^{4^m-1}}$. 
	Let $\sigma$ be a nontrivial embedding of $K(s)$
	into $\bar{K}$. Applying $\sigma$ to \eqref{eq:rewrite zmzj}
	and take the difference, we have:
	\begin{equation}\label{eq:apply sigma}
	\frac{(s-\sigma(s))^{4^m}}{x^{4^m-1}}=
	\frac{c_m}{x^{(4^j-1)q_m/q_j}}(s-\sigma(s))^{4^jq_m/q_j}.
	\end{equation}
	
	Consequently:
	\begin{equation}
	(s-\sigma(s))^{4^m-4^jq_m/q_j}=c_mx^{4^m-1-(4^j-1)q_m/q_j}
	\end{equation}
	
	Raising to the $12$-th power and using \eqref{eq:1-qm/qj},
	we have:
	$$(s-\sigma(s))^{12(4^m-4^jq_m/q_j)}=x^{12(4^m-4^jq_m/q_j)}.$$
	If $4^m-4^jq_m/q_j\neq 0$ then  
	$s-\sigma(s)=\zeta x$ for some $\zeta\in\bar{\bF}_2^*$, hence
	$0=P(s)-P(\sigma(s))=\zeta^4x^4+\zeta^2x^6+\zeta x^4$ which
	is impossible. Therefore we must have 
	$4^m-4^jq_m/q_j=0$, or $q_m/q_j=4^{m-j}$. Together
	with \eqref{eq:rewrite zmzj}, we have:
	$$\left(\frac{c_m}{x^{(4^j-1)q_m/q_j}}-\frac{1}{x^{4^m-1}}\right)s^{4^m}\in K.$$
	Since $s^{4^m}\notin K$, we must have $\displaystyle\frac{c_m}{x^{(4^j-1)q_m/q_j}}-\frac{1}{x^{4^m-1}}=0$. This proves the claim.
	
	From \eqref{eq:rewrite zmzj} and the claim above, we
	have:
	\begin{equation}\label{eq:etam etaj dm}
		\frac{\eta_m}{x^{4^m-1}}=\frac{\eta_j^{4^{m-j}}}{{x^{4^m-1}}}+d_m\ \text{for $m\in\cM$}.
	\end{equation}
	
	Applying Lemma~\ref{lem:eta-eta^4} repeatedly, we have 
	$v(\eta_m-\eta_j^{4^{m-j}})=4^{j+1}-4$. On the other
	hand, \eqref{eq:vdm} gives that 
	$v(x^{4^m-1}d_m)\geq 4^m-1+1-4^{m-j}=4^m-4^{m-j}$. Together
	with \eqref{eq:etam etaj dm}, we have
	$4^{j+1}-4\geq 4^m-4^{m-j}$. This gives a contradiction
	when $m$ is sufficiently large.
	\end{proof}

	\subsection{Proof of Corollary~\ref{cor:A}}
	Notation as in Corollary~\ref{cor:A}, there are at most
	$N(d)$ subextensions $F/K$ of $E/K$. For every such $F/K$, let
	$X(F)$ be the set of elements $t\in E$ integral over 
	$\cO_{K,T}$ such that 
	$\discr_K(t)\in\cO_{K,T}^*$ \emph{and} $K(t)=F$. 
	If $X(F)\neq \emptyset$, pick an element $s\in X(F)$. We
	have that $t\in X(F)$ if and only if 
	$\cO_{K,T}[t]=\cO_{K,T}[s]$. 
	Theorem~\ref{thm:A} gives that there are at most 
	$$q(K)^{d^6}+
		\left(\exp(18^{10})p^{3d^4|T|}\log_p q(K)\right)^{d^3}$$ 
	elements $t_1,\ldots,t_N\in X(F)$
	such that every $t\in X(F)$ has the form 
	$at_i^q+b$ for some $1\leq i\leq N$, power
	$q\geq 1$ of p, $a\in K^*$, and $b\in K$.  By comparing
	discriminant, we have that $a^{d(d-1)}\in \cO_{K,T}^*$, hence
	$a\in \cO_{K,T}^*$. Therefore $b\in \cO_{K,T}$ and this 
	finishes
	the proof.

	\section{Proof of Theorem~\ref{thm:B}}\label{sec:proof of B}
	\subsection{Notation and preliminary results}\label{subsec:notation B} 
	Throughout this section, assume 
	the notation in Theorem~\ref{thm:B}. Recall
	the condition that $\{s^n,t^n:\ n\in\N\}\cap \cO=\emptyset$
	(see Section~\ref{sec:addendum}).
	Let $L$ be
	the Galois closure of $K(s,t)/K$ and let $G$ denote the 
	radical in $L$ of the group generated by $\cO^*$
	and all the conjugates of $s$ and $t$. 
	Let $r$ denote the rank of $G$. Define 
	$e$ (respectively $f$) to be the smallest integer in
	$\N$ such that $K(s^e)\subseteq K(s^n)$ 
	(respectively $K(t^f)\subseteq K(t^n)$) for every $n\in\N$;
	in other words, $K(s^e)=\displaystyle\cap_{n\in \N} K(s^n)$
	(respectively $K(t^f)=\displaystyle\cap_{n\in\N} K(t^n)$). 
	For every  $1\leq k\leq e$ and $1\leq \ell\leq f$,  define 
	the set:
	$$\cM(k,\ell):=\{(m,n)\in \cM(\cO,s,t):\ m\equiv k \bmod e,\ n\equiv \ell
	\bmod f\}.$$
	
	As in the proof of Theorem~\ref{thm:A}, if $\cO[s^m]=\cO[t^n]$ 
	and $\sigma\Gal(L/K(s^m))$
	is a non-identity coset of $\Gal(L/K(s^m))$ in $\Gal(L/K)$,
	then
	$$u_{m,n,\sigma}:=\frac{s^m-\sigma(s^m)}{t^n-\sigma(t^n)}$$
	is a unit in the ring $\cO[s^m,\sigma(s^m)]$.
	
	\begin{lemma}\label{lem:new preserved}
		We have the following:
		\begin{itemize}
			\item [(a)] for every $(m,n)\in \cM(\cO,s,t)$, we have
			$(pm,pn)\in\cM(\cO,s,t)$;
			\item [(b)] $\gcd(e,p)=\gcd(f,p)=1$;
			\item [(c)] there is a power $q_1>1$ of $p$ such that
			 for every $(m,n)\in\cM(k,\ell)$, we have 
			 $(q_1m,q_1n)\in\cM(k,\ell)$.		
		\end{itemize}
	\end{lemma}
	\begin{proof}
		Part (a) follows from the easy fact that if 
		$\cO[s^m]=\cO[t^n]$ then $\cO[s^{pm}]=\cO[t^{pn}]$.
		
		We prove $\gcd(e,p)=1$ by contradiction, the identity 
		$\gcd(f,p)=1$
	could be proved by similar arguments. Assume $e=p\alpha$ with 
	$\alpha\in \N$. By
	the minimality of $e$, we have that $K(s^{\alpha})$ strictly
	contains $K(s^e)$. Pick $\tau\in\Gal(L/K(s^e))$ outside $\Gal(L/K(s^{\alpha}))$. 
	We have: $(s/\tau(s))^e=1$, hence $(s/\tau(s))^{\alpha}=1$ since $e=p\alpha$.  
	Hence $\tau$ fixes $s^{\alpha}$, contradiction. This
	proves part (b).
	
		For part (c), we choose $q_1$ such that
		$q_1k\equiv k$ modulo $e$ and $q_1\ell\equiv \ell$ modulo $f$.
		This is possible by part (b).		
	\end{proof}

	\begin{definition}\label{def:Frobenius}
	We have the following definitions.
	\begin{itemize}
		\item [(a)] Let $M\in\N$, $\bfx\in\N^M$, and $q>1$ be
		a power of $p$. The $q$-Frobenius subset of $\N^M$
		generated by $\bfx$ is defined to be:
		$$F_1(q;\bfx):=\{q^k\bfx:\ k\in\N_0\}.$$
		
		\item [(b)] Let $(a,b)\in\N^2$, the doubly $q$-Frobenius
		subset of $\N^2$ generated by
		$(a,b)$ is defined to be:
		$$F_2(q;a,b):=\{(q^ia,q^jb):\ i,j\in\N_0\}.$$
	\end{itemize}
		We say that $q$ is the \emph{base} of $F_1(q;\bfx)$
		and $F_2(q;a,b)$. A Frobenius subset of $\N^M$ 
		(respectively
		doubly Frobenius subset of $\N^2$)
		is a set of the form $F_1(q;\bfx)$
		(respectively $F_2(q;a,b)$).
	\end{definition}
		
	\begin{remark}
	Sets of the form $F(q;a_1,a_2,a_3,a_4)$
	as defined in Definition~\ref{def:non-degenerate} include
	Frobenius subsets of $\N^2$ (when $a_2=a_4=0$)
	and doubly Frobenius subsets of $\N^2$ (when $a_2=a_3=0$)
	as special cases. 
	\end{remark}

	Our proof of Theorem~\ref{thm:B} will be divided into two 
	cases.
	
	\subsection{The case when $K(s^e)\neq K(t^f)$.}
	In this subsection, we prove the following: 
	\begin{proposition}\label{prop:B2}
	Let $s$ and $t$ be as in Theorem~\ref{thm:B}. Assume that 
	$K(s^e)\neq K(t^f)$. 
	Then the set $\cM:=\cM(\cO,s,t)$ is a finite union of Frobenius and 
	doubly Frobenius subsets of $\N^2$.	
	\end{proposition}
		
	We start with an easy lemma:
	\begin{lemma}\label{lem:easy Frobenius}
	Let $r\in\N$ and let $Z$ be a nonempty subset of $\N^k$. Assume
	that $Z$ is contained in a finite union of Frobenius subsets
	of $\N^k$ and there is $q>1$ which is a power of $p$
	such that $qZ\subset Z$. Then $Z$ is a finite union of
	Frobenius subsets of base $q$.
	\end{lemma}
	\begin{proof}
	We may assume that $Z$ is contained in a finite disjoint
	union of Frobenius subsets of $\N$ whose bases are powers of $q$. 
	Denote
	these Frobenius subsets by
	$F_1,\ldots,F_{n}$. We may assume $Z\cap F_i\neq \emptyset$
	and let $\bfx_i$ be the minimal element in $Z\cap F_i$ for $1\leq i\leq n$.  
	Then we have:
	$$Z=\bigcup_{i=1}^n \{q^n\bfx_i:\ n\in\N_0\}.$$
	\end{proof}

	We have the following:
	\begin{proposition}\label{prop:C1}
    There is a constant $C_1$ such that the following hold. 
		\begin{itemize}
		\item [(a)] For every
			$m\in\pi_1(\cM)$, for every subset
			of at least $C_1$ elements in
			$$\cM_1(m):=\{n\in\N:\ (m,n)\in\cM\}$$
			there exist $n_1<n_2$
			such that $\displaystyle\frac{n_2}{n_1}$ is a power of $p$.
		\item [(b)] For every $n\in\pi_2(\cM)$, 
		for every subset of at least $C_1$ elements in
		$$\cM_2(n):=\{m\in\N:\ (m,n)\in\cM\}$$
		there exist $m_1<m_2$ in $\cM_2(n)$ such that 
		$\displaystyle\frac{m_2}{m_1}$ is a power of $p$.
		\end{itemize}
	\end{proposition}
	\begin{proof}
	It suffices to prove (a) only since (b) is completely analogous. 
	Since $t^k\notin \cO$ for every $k\in\N$, there exists
	$\sigma\in\Gal(L/K)$ such that
	$\sigma$ does not fix $t^k$ for every $k\in\N$.
	Recall that when $(m,n)\in\cM$, we have that 
	$\displaystyle u_{m,n,\sigma}:=\frac{s^m-\sigma(s^m)}{t^n-\sigma(t^n)}$
	is an element of $G$.
	
	Let $\Gamma$ be the group generated by $G$ and 
	$s^m-\sigma(s^m)$.	Hence the rank of $\Gamma$ is
	at most $r+1$.
	We have that $u_{m,n,\sigma}(t^n/(s^m-\sigma(s^m)),-\sigma(t^n)/(s^m-\sigma(s^m)))$ gives a solution of $x+y=1$ with $(x,y)\in \Gamma^2$. By Proposition~\ref{prop:refined 2},
	there is a finite subset $\scrX$ of $L^*$ such that
	$|\scrX|\leq p^{2r+2}$ and 
	$$\frac{t^n}{\sigma(t^n)}=u^{q^{\alpha}}$$
	for some $u\in \scrX$ and $\alpha\in\N_0$. 
	
	Now we can take $C_1=p^{2r+2}+1$.  For any
	$C_1$ distinct elements of $\cM_1(m)$, there are two elements
	$n_1<n_2$ such that there exist $u\in X$ and $\alpha_1<\alpha_2$ such that 
	$$\frac{t^{n_1}}{\sigma(t^{n_1})}=u^{p^{\alpha_1}}\ \text{and }
	\frac{t^{n_2}}{\sigma(t^{n_2})}=u^{p^{\alpha_2}}.$$
	Note that $\frac{t}{\sigma(t)}\notin \bar{\bF}_p^*$
	since $\sigma$ does not fix any power of $t$. Hence $n_1p^{\alpha_2}=n_2p^{\alpha_1}$.
	\end{proof}

\begin{proposition}\label{prop:pi_1 pi_2} The following results hold.
	\begin{itemize} 
	\item [(a)] If $K(s^e)\nsubseteq K(t^f)$ then the set $\pi_1(\cM)$ is a finite union of $p$-Frobenius subsets of $\N$.
	\item [(b)] If $K(t^f)\nsubseteq K(s^e)$ then the set
	$\pi_2(\cM)$ is a finite union of $p$-Frobenius subsets of $\N$.
	\end{itemize}
	\end{proposition}
\begin{proof}
It suffices to prove (a) only.  There exists
$\sigma\in \Gal(L/K(t^f))$ such that $\sigma\notin \Gal(L/K(s^e))$.
For every $1\leq \ell\leq f$,
define 
$\cM(\cdot,\ell):=\{(m,n)\in\cM:\ n\equiv \ell\bmod f\}.$ By
our assumption, we have $\cM(\cdot,f)=\emptyset$.

Fix any $1\leq \ell< f$, let $(m,n)\in \cM(\cdot,\ell)$, and write
$n=\tilde{n}f+\ell$. As before, we have $u_{m,n,\sigma}\in G$ such that:
	$$0\neq s^m-\sigma(s^m)=u_{m,n,\sigma} (t^n-\sigma(t^n))=u_{m,n,\sigma}t^{\tilde{n}f}(t^{\ell}-\sigma(t^{\ell})).$$
Let $\Gamma$ be the group generated by $G$ and $t^{\ell}-\sigma(t^{\ell})$ whose rank is at most $r+1$.
The above identity gives that $\displaystyle\frac{1}{u_{m,n,\sigma}t^{\tilde{n}f}(t^{\ell}-\sigma(t^{\ell}))}(s^m,-\sigma(s^m))$
is a solution of the equation $x+y=1$ with $(x,y)\in G^2$. Note that
$\displaystyle\frac{s}{\sigma(s)}\notin \bar{\bF}_p^*$ since $\sigma$ cannot fix any power of $s$. Write $C_2=p^{2r+2}+1$.
 By using Proposition~\ref{prop:refined 2} as in the proof of Proposition~\ref{prop:C1},
we have that for every subset of at least $C_2$ elements of 
$\pi_1(W(\cdot,\ell))$,
there exist $m_1<m_2$  such that
$\displaystyle\frac{m_2}{m_1}$
is a power of $p$. Hence the same conclusion holds for every subset of at least $fC_2$ elements of $\pi_1(W)$. 
By Lemma~\ref{lem:new preserved}, if $m\in\pi_1(\cM)$ then $pm\in \pi_1(\cM)$. This implies that $\pi(W)$ is the union of
at most $fC_2$ many $p$-Frobenius subsets.
\end{proof}
	
	\begin{proof}[Proof of Proposition~\ref{prop:B2}]
We may assume $K(s^e)\nsubseteq K(t^f)$
since the case $K(t^f)\nsubseteq K(s^e)$ is similar. By Proposition~\ref{prop:pi_1 pi_2}, the set $\pi_1(\cM)$ is
a (disjoint) union of finitely many $p$-Frobenius subsets $F_1,\ldots,F_k$
of $\N$. Fix any $i$ such that $1\leq i\leq k$,
it suffices to show that 
the set:
$$\cM\cap F_i\times \N=\{(m,n)\in \cM:\ m\in F_i\}$$
is contained in finitely many
Frobenius and doubly Frobenius subsets \emph{of $\cM$}.

Recall the notation $\cM_1(m)$ and the constant $C_1$ in Proposition~\ref{prop:C1}. There are two cases:

\textbf{Case 1:} $\{|\cM_1(m)|:\ m\in F_i\}$ is bounded. Let $m_i\in F_i$
be such that $M:=|\cM_1(m_i)|=\max\{|\cM_1(m)|:\ m\in F_i\}$.
Denote 
$\cM_1(m_i):=\{n_1,\ldots,n_M\}.$
For every $m\in F_i$ satisfying $m>m_i$, write $m=p^vm_i$ for some $v\in\N$. By Lemma~\ref{lem:new preserved} 
 and the maximality of $|\cM_1(m_i)|$, we
conclude that:
$$\cM_1(m)=\{p^vn_1,\ldots,p^vn_M\}.$$
Hence the set $\{(m,n)\in \cM\cap F_i\times \N:\ m\geq m_i\}$ 
is a finite union of $p$-Frobenius subsets. The
set $\{(m,n)\in \cM\cap F_i\times\N:\ m<m_i\}$ is finite by our assumption in this case, hence, by Lemma~\ref{lem:new preserved}, 
it is contained in a finite 
union of Frobenius subsets of $\cM\cap F_i\times\N$. Overall, we have that $\cM\cap F_i\times\N$
is a finite union of its Frobenius subsets.

\textbf{Case 2:} $\{|\cM_1(m)|:\ m\in F_i\}$ is unbounded. Hence there exists $\tilde{m}\in F_i$, chosen to be minimal, 
such that $|\cM_1(\tilde{m})|>C_1$. By Proposition~\ref{prop:C1}, 
 there
are $n_1<n_2$ in $\cM_1(\tilde{m})$
such that $q:=\displaystyle\frac{n_2}{n_1}$ is a power of $p$. This implies
$\cO[s^{\tilde{m}}]=\cO[t^{n_1}]=\cO[t^{n_2}]=\cO[t^{n_1q}]
=\cO[s^{\tilde{m}q}].$

Hence
for every $n\in \cM_1(\tilde{m})$, we have $\cO[t^{nq}]=\cO[s^{\tilde{m}q}]
=\cO[s^{\tilde{m}}]$ which gives that $nq\in\cM_1(\tilde{m})$. By 
Proposition~\ref{prop:pi_1 pi_2} and
Lemma~\ref{lem:easy Frobenius},
 $\cM_1(\tilde{m})$ is a finite union of Frobenius subsets of base $q$. Since $\cO[s^{\tilde{m}}]=\cO[s^{\tilde{m}q}]$,
 we have that $\cM_1(\tilde{m})=\cM_1(\tilde{m}q^v)$ for every
 $v\in\N$. Therefore, the set:
 $$\{(m,n)\in \cM\cap F_i\times\N:\ m=\tilde{m}q^v\ \text{for some $v\in\N_0$}\}$$
 is a finite union of doubly Frobenius subsets (of base $q$). 
 
Write $q=p^{w}$ for some $w\in\N$. For $1\leq j\leq w-1$, by Lemma~\ref{lem:new preserved}, the set $\cM_1(\tilde{m}p^j)$ has two elements whose
quotient is $q$ since the same holds for $\cM_1(\tilde{m})$. Hence we repeat
the same arguments where $\tilde{m}$ is replaced by $\tilde{m}p^j$
to conclude that the set:
 $$\{(m,n)\in \cM\cap F_i\times\N:\ m=\tilde{m}p^jq^v\ \text{for some $v\in\N_0$}\}$$
 is a finite union of doubly Frobenius subsets (of base $q$).  

Finally, by the minimality of $\tilde{m}$, the set
$$\{(m,n)\in \cM\cap F_i\times\N: m<\tilde{m}\}$$
is finite. By Lemma~\ref{lem:new preserved}, this set is contained
in a finite union of Frobenius subsets of $\cM\cap F_i\times\N$. 

Overall, we conclude that $\cM\cap F_i\times\N$ is a finite union
of Frobenius and doubly Frobenius subsets. This finishes the
proof.
	\end{proof}

	\subsection{The case when $K(s^e)=K(t^f)$}
	We now assume that $K(s^e)=K(t^f)$ and denote this field by $K^o$. Note that $K\subsetneq K^o$ by the assumption on $s$ and $t$. For  $1\leq k\leq e$ and $1\leq \ell\leq f$,  consider the set
	$\cM(k,\ell)=\{(m,n)\in \cM(\cO,s,t):\ m\equiv k \bmod e,\ n\equiv \ell
	\bmod f\}.$
	The convenience
	of doing this is that we can fix $F:=K(s^k)=K(s^m)=K(t^n)=K(t^{\ell})$
	for $(m,n)\in \cM(k,\ell)$. We have the tower of fields:
	  $$K\subsetneq K^o\subset F\subset L.$$

	As before, for every $(m,n)\in \cM(k,\ell)$ and
	 $\sigma\in \Gal(L/K)\setminus \Gal(L/F)$ there is $u_{m,n,\sigma}\in G$ such that
	$0\neq s^m-\sigma(s^m)=u_{m,n,\sigma}(t^n-\sigma(t^n)).$
	Therefore $\displaystyle\bfx_{m,n,\sigma}:=
	\left(\frac{s^m}{\sigma(s^m)},-\frac{u_{m,n,\sigma}t^n}{\sigma(s^m)},
	\frac{u_{m,n,\sigma}\sigma(t^n)}{\sigma(s^m)}\right)$ is a solution of
	the unit equation
	\begin{equation}\label{eq:3 variables}
	x+y+z=1\ \text{with $(x,y,z)\in G^3$.} 
	\end{equation}
	Note that $\bfx_{m,n,\sigma}=\bfx_{m,n,\tau}$
	and $u_{m,n,\sigma}=u_{m,n,\tau}$
	if the two cosets $\sigma \Gal(L/F)$ and $\tau\Gal(L/F)$ coincide. We have the following:
	\begin{proposition}\label{prop:new degenerate}
	The set of $(m,n)\in\cM(\cO,s,t)$ such that $\bfx_{m,n,\sigma}$
	is degenerate
	for every coset $\sigma\Gal(L/F)$ with $\sigma\notin\Gal(L/K^o)$
	is contained in $\cA(\cO,s,t)\cup\cB(\cO,s,t)\cup\cC(\cO,s,t)$.
	\end{proposition}
	\begin{proof}
		A proof in the characteristic zero case is given in
		\cite[pp.~12--14]{KN-TAMS2015}. The same proof
		can be used for positive characteristic.
	\end{proof}

	By Proposition~\ref{prop:new degenerate}, to finish
	the proof of Theorem~\ref{thm:B}, we show that 	
	for every $\sigma\in\Gal(L/K)\setminus\Gal(L/K^o)$,  
	the set:
	$$\cM(k,\ell,\sigma):=\{(m,n)\in \cM(k,\ell):\ \text{$\bfx_{m,n,\sigma}$ is 
	nondegenerate}\}$$
	is contained in a finite union of sets of the form
	$F(q;c_1,c_2,c_3,c_4)$. We have:
	\begin{lemma}\label{lem:preserve degeneracy}
	There is a power $q_1>1$ of $p$ such that 
	$(q_1m,q_1n)\in\cM(k,\ell,\sigma)$ whenever
	 $(m,n)\in \cM(k,\ell,\sigma)$.
	\end{lemma}
	\begin{proof}
	This follows from part (c) of Lemma~\ref{lem:new preserved} and the fact that $x_{q_1m,q_1n,\sigma}$ is degenerate iff $x_{m,n,\sigma}$ is degenerate.
	\end{proof}

	Since $\sigma\notin\Gal(L/K^o)$,
	it does not fix $s^n$ or $t^n$ for any $n\in\N$. In other words,
	$\displaystyle \frac{s}{\sigma(s)}$
	and $\displaystyle \frac{t}{\sigma(t)}$
	are not roots of unity.
	Apply Proposition~\ref{prop:refined n} to
	the unit equation \eqref{eq:3 variables}, we have that there
	exists a positive integer $c$ and a finite set 
	$\scrS'$ (contained in $\bar{L}^*$)
	such that for every $(m,n)\in\cM(k,\ell,\sigma)$ the
	identities
	\begin{equation}\label{eq:apply unit eq}
		\left(\frac{s^m}{\sigma(s^m)}\right)^{p^c}=x_1^{p^i}x_2^{p^j};\ 
		\left(\frac{-u_{m,n,\sigma}t^n}{\sigma(s^m)}\right)^{p^c}=y_1^{p^i}y_2^{p^j};\ 
		\left(\frac{u_{m,n,\sigma}\sigma(t^n)}{\sigma(s^m)}\right)^{p^c}=z_1^{p^i}z_2^{p^j}	
	\end{equation}
	
	
	hold for some $x_1,\ldots,z_2\in \scrS'$
	and some $i,j\in\N_0$. Note that the last two equations
	of \eqref{eq:apply unit eq}
	implies that $\displaystyle\left(\frac{t^n}{\sigma(t^n)}\right)^{p^c}$
	also has the form $w_1^{p^i}w_2^{p^j}$
	where $w_1$ and $w_2$ belong to a finite set. Let $\tilde{G}$
	be the group
 generated by this finite set, 
	the set $\scrS'$, and the group $G$. Since
	$\displaystyle\frac{s}{\sigma(s)}$
	and $\displaystyle\frac{t}{\sigma(t)}$
	are non-torsion, using a basis of the free group
	$\tilde{G}/\tilde{G}_{\tor}$ to compare exponents as
	in the proof of Proposition~\ref{prop:refined n},
	we have that there exist finitely
	many quadruples 
	$\bfa_h=(a_{h1},a_{h2},a_{h3},a_{h4})\in\Q^4$
	for $h\in I$
	such that $\scrM(k,\ell,\sigma)$ is contained in
	$$\displaystyle\bigcup_{h\in I}F(p;a_{h1},a_{h2},a_{h3},a_{h4})$$
	where (recall Definition~\ref{def:non-degenerate})		
	$F(p;a_{h1},\ldots,a_{h4})=\{(a_{h1}p^i+a_{h2}p^j,a_{h3}p^i+a_{h4}p^j):\ i,j\in\N_0\}$. This finishes the proof of
	Theorem~\ref{thm:B}.

	\section{An addendum to Theorem~\ref{thm:B}}\label{sec:addendum}
	For the sake of completeness, we briefly discuss Problem~(B)
	under the condition that $\{s^n,t^n:\ n\in\N\}\cap\cO\neq \emptyset$. The problem in this case becomes much easier and
	we model this section based on \cite[Section~5]{KN-TAMS2015}
	with appropriate modification for positive characteristic. 
	 Write $\cM:=\cM(\cO,s,t)$. For $\alpha,\beta\in \N$,
	let $A(\alpha,\beta)$ denote the arithmetic progression
	$\{k\alpha+\beta:\ k\in\N_0\}$.
	We may assume $t^f\in \cO$ and consider two cases.
	
	\subsection{The case $s^e\notin\cO$}
	For $1\leq \ell\leq f$, let
	$\cM(\cdot,\ell):=\{(m,n)\in\cM(\cO,s,t):\ n\equiv \ell \bmod f\}$. Note that $\cM(\cdot,f)=\emptyset$ since $K(s^e)\neq K$. We have:
	\begin{proposition}\label{prop:s^e notin O} The following results hold.
		\begin{itemize}
			\item [(a)] For each $\ell\in\{1,\ldots,f-1\}$,
			$\pi_1(\cM(\cdot,\ell))$ is a finite union 
			of Frobenius
			subsets of $\N$. 
			\item [(b)] If $t^f\notin\cO^*$ then $\cM$
			is a finite union of Frobenius subsets of $\N^2$.
			\item [(c)] Assume $t^f\in\cO^*$ and let $\ell\in\{1,\ldots,f-1\}$. Then 
			$\cM(\cdot,\ell)=\pi_1(\cM(\cdot,\ell))\times A(f,\ell)$.		
		\end{itemize}
	\end{proposition}
	\begin{proof}
	 The same arguments in the proof of Proposition~\ref{prop:pi_1 pi_2} can be used to prove part (a). 
	 
	 For part (b), note that $\pi_1(\cM)$ is a finite union of
	 Frobenius subsets of $\N$ due to part (a). For any $m\in\pi_1(\cM)$, we prove that the set 
	 $$\cM_1(m):=\{n\in\N:\ (m,n)\in \cM_1(m)\}$$
	has at most $f-1$ elements. Once this is done, we can use the same
	arguments as in the first case of the proof of Proposition~\ref{prop:B2}. 
	It suffices to show that for any $\ell\in\{1,\ldots,f-1\}$, there is at most one $n\in\N$ such that 
	$(m,n)\in \cM$ and $n\equiv \ell \bmod f$. Assume
	there are two such elements, namely $n_1<n_2$.
	Write $n_1=\tilde{n}_1f+\ell$ and $n_2=\tilde{n}_2f+\ell$.
	Pick
	$\sigma\in \Gal(L/K)$ such that $\sigma\notin\Gal(L/K(s^e))$,
	hence $\sigma$ does not fix any power of $s$. As before,
	there are units $u_{m,n_1,\sigma}$ and $u_{m,n_2,\sigma}$
	in $\cO[s^m,\sigma(s^m)]$
	such that:
	$$0\neq s^m-\sigma(s^m)=u_{m,n_1,\sigma}(t^{n_1}-\sigma(t^{n_1}))=u_{m,n_1,\sigma}t^{\tilde{n}_1f}(t^{\ell}-\sigma(t^{\ell}));$$ 
	$$0\neq s^m-\sigma(s^m)=u_{m,n_2,\sigma}(t^{n_2}-\sigma(t^{n_2}))=u_{m,n_2,\sigma}t^{\tilde{n}_2f}(t^{\ell}-\sigma(t^{\ell})).$$
	This implies $t^{(\tilde{n}_2-\tilde{n}_1)f}$ is a unit. Hence $t^f\in\cO^*$, contradiction. This proves part (b).
	
	Part (c) follows from the fact that  $\cO[t^n]=\cO[t^{\ell}]$
	for every $n\in A(f,\ell)$ since $t^f\in\cO^*$.
	\end{proof}
	
	\subsection{The case $s^e\in\cO$}
	This could be taken almost verbatim from \cite[Subsection~5.2]{KN-TAMS2015}, so we will be brief. It suffices to describe
	the set $\cM(k,\ell)$ for $1\leq k\leq e$ and 
	$1\leq \ell\leq f$. It is immediate that 
	$W(e,\ell)=W(k,f)=\emptyset$ if $\ell<f$ and $k<e$. On
	the other hand, $W(e,f)=e\N\times f\N$.
	
	From now on, we study $W(k,\ell)$ under the assumption
	that $k<e$ and $\ell<f$. We also assume
	$K(s^k)=K(t^{\ell})$, otherwise $W(k,\ell)=\emptyset$.
	By the same arguments in \cite[pp.~15]{KN-TAMS2015},
	we have that if there exist distinct $(m_1,n_1),(m_2,n_2)\in 
	W(k,\ell)$ then:
	\begin{equation}\label{eq:m2m1 n2n1}
		\frac{s^{m_2-m_1}}{t^{n_2-n_1}}\in \cO^*.
	\end{equation}
	
	We have the following:
		\begin{proposition}\label{prop:non-unit} The following results hold.
		\begin{itemize}
		\item [(a)] If both $s$ and $t$ are units (i.e. $s,t\in\cO_L^*$) then
		$W(k,\ell)$ is either empty or has the form
		$$(k,\ell)+e\N\times f\N.$$
		\item [(b)] If $s$ is a unit and $t$ is not then 
		$W(k,\ell)$
		is empty.
		The similar statement holds
		when $t$ is a unit and $s$ is not.
		\item [(c)] Assume that neither $s$ nor $t$ is a unit. 
		If $W(k,\ell)\neq \emptyset$
		then the following holds. There is a minimal pair $(M,N)\in\N^2$
		satisfying $\displaystyle\frac{s^{eM}}{t^{fN}}\in\cO^*$. For any two distinct
		pairs
		$(m_1,n_1),(m_2,n_2)\in W(k,\ell)$, we have $(m_2-m_1)(n_2-n_1)>0$.
		Moreover, we have $\displaystyle\frac{m_2-m_1}{eM}=\frac{n_2-n_1}{fN}$ and it is an integer.
		\end{itemize}	
	\end{proposition}
	\begin{proof}
		This is proved as in the proof of \cite[Proposition~5.2]
		{KN-TAMS2015}. The only difference is that in our current 
		setting, if $W(k,\ell)\neq \emptyset$ then it is infinite.
		In fact, let $q>1$ be a power of $p$ such that 
		$q\equiv 1 \bmod ef$. We have that $(qm,qn)\in W(k,\ell)$
		whenever $(m,n)\in W(k,\ell)$. This fact and Equation \eqref{eq:m2m1 n2n1} imply part (b).
	\end{proof}
	
	From now on, we assume that neither $s$ nor $t$ is a unit,
	there is a minimal pair $(M,N)\in\N^2$ such that
	$\displaystyle \frac{s^{eM}}{t^{fN}}\in\cO^*$, and 
	$W(k,\ell)\neq\emptyset$. Part (c) of Proposition~
	\ref{prop:non-unit} implies that
	$W(k,\ell)$ has
	a minimal element $(\tm,\tn)$ and every $(m,n)\in W(k,\ell)$
	has the form $(\tm+\delta eM,\tn+\delta fN)$
	for some $\delta\in\N_0$. We can use exactly the same
	method in \cite[pp.16--17]{KN-TAMS2015}
	to find an upper bound for $(\tm,\tn)$ and 
	to determine all such $\delta$'s.

	\bibliographystyle{amsalpha}
	\bibliography{Char_p_30Aug2015}

\end{document}